\documentclass[11pt]{amsart}

\usepackage{amsfonts}
\usepackage{amsmath}
\usepackage{amssymb}
\usepackage{amsthm}
\usepackage{amsxtra}
\usepackage{epic, eepic}
\usepackage{epsfig,color}
\usepackage{fullpage}
\usepackage{vmargin}
\theoremstyle{plain}

\newtheorem{Thm}{Theorem}[section]

\newtheorem{Lem}[Thm]{Lemma}

\newtheorem*{Claim}{Claim}

\theoremstyle{definition}

\theoremstyle{remark}
\newtheorem{Rem}[Thm]{Remark}

\newtheorem{Ex}[Thm]{Example}

\numberwithin{equation}{section}

\setmarginsrb{33mm}{30mm}{33mm}{40mm}%
            {0mm}{10mm}{0mm}{10mm}

\begin{document}

\newcommand{\rx}{\mathcal{R}(X)}

\title{Arithmetic matroids, Tutte polynomial\\ and toric arrangements}

\author{Michele D'Adderio$^*$}\thanks{$^*$ supported by the Max-Planck-Institut
f\"{u}r Mathematik and the University of
G\"{o}ttingen.}\author{Luca Moci$^{\dag}$}\thanks{$^{\dag}$
supported by Hausdorff Institute (Bonn), Institut Mittag-Leffler (Djursholm), and Dipartimento di Matematica "Guido Castelnuovo" (Roma).}
\address{Georg-August Universit\"{a}t G\"{o}ttingen\\ Mathematisches Institut\\
Bunsenstrasse 3-5, D-37073 G\"{o}ttingen\\
Germany}\email{mdadderio@yahoo.it}

\address{Dipartimento di Matematica "Guido Castelnuovo"\\Sapienza Universit\`a di Roma\\
Piazzale Aldo Moro 5, 00185 Roma\\
Italy}\email{moci@mat.uniroma1.it}

\begin{abstract}
We introduce the notion of an arithmetic matroid, whose main
example is given by a list of elements of a finitely
generated abelian group. In particular, we study the
representability of its dual, providing an extension of the Gale
duality to this setting.

Guided by the geometry of generalized toric arrangements, we
provide a combinatorial interpretation of the associated
arithmetic Tutte polynomial, which can be seen as a generalization
of Crapo's formula for the classical Tutte polynomial.
\end{abstract}

\maketitle

\section*{Introduction}

Who can imagine a simpler object than a finite list of vectors?

Nevertheless, several mathematical constructions arise from such a
list $X$: hyperplane arrangements and zonotopes in geometry, box
splines in numerical analysis, root systems and parking
functions in combinatorics are only some of the most well-known
examples. More recently, Holtz and Ron in \cite{HR}, and Ardila
and Postnikov in \cite{AP} introduced various algebraic structures
capturing a rich description of these objects.

A central role in this framework is played by the combinatorial
notion of \emph{matroid}, which axiomatizes the linear dependence
of the elements of $X$.

If the list $X$ lies in $\mathbb{Z}^n$, an even wider spectrum of
mathematical objects appears. In their recent book
\cite{li}, De Concini and Procesi explored (among other things)
the connection between the \textit{toric arrangement} associated
to such a list and the \emph{vector partition function}. Inspired
by earlier work of Dahmen and Micchelli (\cite{DM2}, \cite{DMi}),
they view this relation as the discrete analogue of the one
between hyperplane arrangement and multivariate spline. This
approach has also surprising applications to the equivariant index
theory (\cite{DPV1}, \cite{DPV2}, \cite{DPV5}). While the spline
and the hyperplane arrangement only depend on the ``linear
algebra'' of $X$, the partition function and the toric arrangement
are also influenced by its ``arithmetics''.

In fact, in order to have effective inductive methods, one needs
to enlarge the picture from $\mathbb{Z}^n$ to its possible
quotients, i.e. finitely generated abelian groups.

In this paper we introduce the notion of an \emph{arithmetic
matroid}: this is going to be a matroid $\mathfrak{M}$ together
with a \textit{multiplicity function} $m$ (see the definition in
Section 1.3). This object axiomatizes both the linear algebra (via
the matroid) and the arithmetics (via the multiplicity function)
of a list of elements in a finitely generated abelian group. When
an arithmetic matroid actually comes from such a list we will say
that it is \emph{representable}.

We introduce also the notion of a \textit{dual} arithmetic
matroid, and show that the dual of a representable matroid is
representable.  We provide an explicit construction that extends
the Gale duality to our setting
 (\textbf{Theorem \ref{thm:dual}}).

To every arithmetic matroid $(\mathfrak{M},m)$ we associate an
\emph{arithmetic Tutte polynomial}
$$
M(x,y):=\sum_{A\subseteq
X}m(A)(x-1)^{rk(X)-rk(A)}(y-1)^{|A|-rk(A)},
$$
where $X$ is the list of vectors of $\mathfrak{M}$. When the
multiplicity function $m$ is trivial (i.e. $m\equiv 1$), this
gives the classical Tutte polynomial of the underlying matroid
$\mathfrak{M}$. When $(\mathfrak{M},m)$ is representable, this is
the polynomial defined in \cite{MoT}, where it is shown to have
several applications to vector partition functions, toric
arrangements and zonotopes (cf. also \cite{DM}).

For representable arithmetic matroids, the positivity of the
coefficients of the associated arithmetic Tutte polynomial was
established in \cite{MoT}, while the understanding of their meaning was left as an open problem.

The main result of this paper is to provide a combinatorial
interpretation of the arithmetic Tutte polynomial of \textit{any}
arithmetic matroid, showing in particular the positivity of its
coefficients (see \textbf{Theorem
\ref{thm:generalcase}}, and Section 5.2-5.3 for related definitions). Our interpretation can be seen
as an extension of the one given by Crapo in \cite{Cr} for the
classical Tutte polynomial (see Section 3.1): in fact, when the
multiplicity function is trivial, we recover exactly Crapo's
formula.

Our combinatorial ideas have their roots in the notion of a
\textit{generalized toric arrangement}, which provides the
geometric inspiration and motivation of our work (see Section 4).

\bigskip

The paper is organized in the following way.

In the first section we give the definition of an arithmetic
matroid, we introduce the notion of the dual and representability,
and provide examples of both representable (Section 1.4) and
non-representable (Section 1.5) arithmetic matroids.

In the second section we prove that the dual of a representable
matroid is representable.

In the third section we introduce the arithmetic Tutte polynomial
and several other basic notions and constructions.

The fourth section provides motivational background: it is an
overview on generalized toric arrangements, which are the
geometric counterpart of the combinatorics developed in this
paper.

This inspires and motivates the definitions given in the fifth
section, in which we introduce the main ingredients of our
construction.

We then provide the combinatorial interpretation of the arithmetic
Tutte polynomial: in the basic case of \textit{molecules} (sixth section), and in the general case (seventh section).

Finally in the last section we make a remark on log-concavity and unimodality.

\subsection*{Acknowledgments}
We are deeply grateful to Petter Br\"{a}nd\'{e}n for having
pointed out a mistake in an earlier choice of the axioms in
Section 1.3, and for having suggested how to fix it. We are
very grateful to Emanuele Delucchi for several stimulating and
insightful discussions. We would also like to thank Anders
Bj\"{o}rner, Corrado De Concini, Alicia Dickenstein, Michael Falk,
Matthias Lenz, Alex Postnikov and Frank Sottile for many
interesting conversations and valuable suggestions.

\section{Arithmetic matroids: definitions and examples}

\subsection{Some notation}

We will use the word \emph{list} as a synonymous of multiset.
Hence a list may contain several copies of the same element.

We will use set-theoretic notation such as $A\subseteq X$ to say
that $A$ is a sublist of $X$, $A\cup B$ to denote the merge of two
sublists $A, B\subseteq X$, or $A=\emptyset$ to denote the empty
list. In particular with $A\cap B$ we mean the intersection as
sublist. Hence for example if in $X$ there are two distinct copies
of the same element, one appearing only in $A$ and the other only
in $B$, the intersection of the two sublists does not contain any
of the two copies, although the set-theoretic intersection
contains the element. Similarly for $A\setminus B$. By abuse of
notation we sometimes denote a list with curly brackets, instead
of using the more appropriate round bracket notation.

Given a list $X$, $\mathbb{P}(X)$ is the \textit{power set} of
$X$, i.e. the set of all sublists (including the empty list) of
$X$.

\subsection{Classical matroids}
A \textit{matroid} $\mathfrak{M}=\mathfrak{M}_X=(X,I)$ is a list
of \textit{vectors} $X$ with a set $I\subseteq \mathbb{P}(X)$,
whose elements are called \textit{independent sublists},
satisfying the following axioms:
\begin{enumerate}
  \item the empty list is independent;
  \item every sublist of an independent sublist is independent;
  \item let $A$ and $B$ be two independent sublists and assume
      that $A$ has more elements than $B$. Then there exists
      an element $a\in A\setminus B$ such that $B\cup\{a\}$
      is still independent.
\end{enumerate}

A maximal independent sublist is called a \emph{basis}. It easily
follows from the axioms that the collection of the bases
determines the matroid structure.

The last axiom implies that all the bases have the same
cardinality, which is called the \emph{rank} of the matroid.

Recall that $\mathfrak{M}$ is equipped with a \textit{rank
function} $rk:\mathbb{P}(X)\rightarrow \mathbb{N}\cup \{0\}$,
which is defined by $rk(A):=$ the maximal cardinality of an
independent sublist of $A$, for every $A\in \mathbb{P}(X)$. Notice
that the independent sublists are precisely the sublists whose
cardinality equals the rank. So the rank function determines the
matroid structure.

The axioms of a matroid can be given in several ways (see \cite{Ox}). We state them in terms of the rank
function, since they turn out to be more suitable for our work.

A \textit{matroid} $\mathfrak{M}=\mathfrak{M}_X=(X,rk)$ is a list
of \textit{vectors} $X$ with a \textit{rank function}
$rk:\mathbb{P}(X)\rightarrow \mathbb{N}\cup \{0\}$ which satisfies
the following axioms:
\begin{enumerate}
    \item if $A\subseteq X$, then $rk(A)\leq |A|$;
    \item if $A,B\subseteq X$ and $A\subseteq B$, then $rk(A)\leq rk(B)$;
    \item if $A,B\subseteq X$, then $rk(A\cup B)+rk(A\cap B)\leq
    rk(A)+rk(B)$.
\end{enumerate}
Notice in particular that the first axiom implies that
$rk(\emptyset)=0$.

The \textit{dual} of the matroid $\mathfrak{M}=(X,I)$ is defined
as the matroid with the same list $X$ of vectors, and with bases
the complements of the bases of $\mathfrak{M}$. We will denote it
by $\mathfrak{M}^*$. Notice that the rank function of
$\mathfrak{M}^*$ is given by $rk^*(A):=|A|-rk(X)+rk(X\setminus
A)$. In particular the rank of $\mathfrak{M}^*$ is the cardinality
of $X$ minus the rank of $\mathfrak{M}$.

We say that $v\in X$ is \textit{dependent} on $A\subseteq X$ if
$rk(A\cup\{v\})=rk(A)$, while we say that $v\in X$ is
\textit{independent} on $A$ if $rk(A\cup\{v\})=rk(A)+1$.

\subsection{Arithmetic matroids}

An \textit{arithmetic matroid} is a pair $(\mathfrak{M}_X,m)$,
where $\mathfrak{M}_X$ is a matroid on a list of vectors $X$, and
$m$ is a \textit{multiplicity function}, i.e.
$m:\mathbb{P}(X)\rightarrow \mathbb{N}\setminus \{0\}$ has the
following properties:
\begin{itemize}
    \item[(1)] if $A\subseteq X$ and $v\in X$ is dependent on $A$, then
    $m(A\cup\{v\})$ divides $m(A)$;
    \item[(2)] if $A\subseteq X$ and $v\in X$ is independent on $A$, then
    $m(A)$ divides $m(A\cup\{v\})$;
    \item[(3)] if $A\subseteq B\subseteq X$ and $B$ is a disjoint union $B=A\cup F\cup T$
    such that for all $A\subseteq C\subseteq B$  we have $rk(C)=rk(A)+|C\cap F|$, then
    $$
    m(A)\cdot m(B) = m(A\cup F)\cdot m(A\cup T).
    $$
    \item[(4)] if $A\subseteq B\subseteq X$ and $rk(A)=rk(B)$, then
    $$
    \mu_B(A):=\sum_{A\subseteq T\subseteq B}(-1)^{|T|-|A|}m(T)\geq 0.
    $$
    \item[(5)] if $A\subseteq B\subseteq X$ and $rk^*(A)=rk^*(B)$, then
    $$
    \mu_{B}^*(A):=\sum_{A\subseteq T\subseteq B}(-1)^{|T|-|A|}m(X\setminus T)\geq 0.
    $$
\end{itemize}
When $B=X$ we will denote $\mu_B(A)$ and $\mu_{B}^*(A)$ by
$\mu(A)$ and $\mu^*(A)$ respectively.

We will discuss further these axioms in Remarks
\ref{rm:axioms} and \ref{rm:axiom3}.

\begin{Rem}
The idea of enriching the matroid structure with a multiplicity
function was hinted in \cite{MoT}. However no axioms were given
for this function, so the concept remained vague. We have chosen
the name ``arithmetic matroid'' to avoid confusion with previous
constructions, and to emphasize the meaning of the multiplicity
function.
\end{Rem}

By abuse of notation, we sometimes denote by $\mathfrak{M}$ both
the arithmetic matroid and its underlying matroid.

We define the \textit{dual} of an arithmetic matroid
$(\mathfrak{M}_X,m)$ as the pair $(\mathfrak{M}_{X}^*,m^*)$, where
$\mathfrak{M}_{X}^*$ is the dual of $\mathfrak{M}_X$, and for all
$A\subseteq X$ we set $m^*(A):=m(X\setminus A)$. The following
lemma shows that this is in fact an arithmetic matroid.
\begin{Lem}
The dual of an arithmetic matroid is an arithmetic matroid.
\end{Lem}
\begin{proof}
We need to show that $m^*$ is a multiplicity function. It is
immediate to check that the axioms (1) and (2) are dual to each
others, i.e. the axiom (2) is equivalent to
\begin{enumerate}
    \item[($1^*$)] if $A\subseteq X$ and $v\in X$ is independent on $A$ in the
    dual, i.e. $rk^*(A\cup \{v\})=rk^*(A)+1$, then
    $m^*(A)$ divides $m^*(A\cup\{v\})$,
\end{enumerate}
while axiom (1) is equivalent to
\begin{enumerate}
    \item[($2^*$)] if $A\subseteq X$ and $v\in X$ is dependent on $A$ in the
    dual, i.e. $rk^*(A\cup \{v\})=rk^*(A)$, then
    $m^*(A\cup\{v\})$ divides $m^*(A)$.
\end{enumerate}
So clearly they are both satisfied in the dual.

The same is clearly true for axioms (4) and (5).

To check the third axiom, we notice that it is in fact
``self-dual''. More precisely, let $A\subseteq B\subseteq X$ and
$B$ be a disjoint union $B=A\cup F\cup T$ such that for all
$A\subseteq C\subseteq B$ we have $rk^*(C)=rk^*(A)+|C\cap F|$.
Notice that for $C:=B$ this implies $rk^*(B)=rk^*(A)+|F|$, i.e.
$$
|B|-rk(X)+rk(X\setminus B)=|A|-rk(X)+rk(X\setminus A)+|F|
$$
or
\begin{equation} \label{eq:lemdual}
rk(X\setminus A)=rk(X\setminus B)+|B|-|A|-|F|=rk(X\setminus
B)+|T|.
\end{equation}

Also, $X\setminus A$ is a disjoint union $X\setminus A=(X\setminus
B)\cup T\cup F$. For $A\subseteq C\subseteq B$, we have
$(X\setminus B)\subseteq (X\setminus C)\subseteq (X\setminus A)$.
We want to show that $rk(X\setminus C)=rk(X\setminus
B)+|(X\setminus C)\cap T|$, so that we are in the hypothesis of
axiom (3), and therefore we get
\begin{eqnarray*}
m^*(A)\cdot m^*(B) & = & m(X\setminus A)\cdot m(X\setminus
B)\\
\text{(by axiom (3))} & = & m((X\setminus B)\cup F)\cdot
m((X\setminus B)\cup T)\\
 & = & m((X\setminus B)\cup (X\setminus (X\setminus F)))\cdot
m((X\setminus B)\cup (X\setminus (X\setminus T)))\\
 & = & m(X\setminus (B\cap (X\setminus F)))\cdot
m(X\setminus (B\cap (X\setminus T)))\\
 & = & m(X\setminus (A\cup T))\cdot
m(X\setminus (A\cup F))\\
 & = & m^*(A\cup T)\cdot m^*(A\cup F).
\end{eqnarray*}

We have
\begin{eqnarray*}
|C|-rk(X)+rk(X\setminus C) & = & rk^*(C)\\
 & = & rk^*(A)+|C\cap F|\\
 & = & |A|-rk(X)+rk(X\setminus
A)+|C\cap F|,
\end{eqnarray*}
and this implies
\begin{eqnarray*}
rk(X\setminus C) & = & rk(X\setminus A)+|A|+|C\cap
F|-|C|\\
 & = & rk(X\setminus A)-|C\cap T|\\
 \text{(by \eqref{eq:lemdual})} & = & rk(X\setminus B)+|T|-|C\cap T|\\
 & = & rk(X\setminus B) + |(X\setminus C)\cap T|,
\end{eqnarray*}
as we wanted.
\end{proof}
\begin{Rem}
Notice that setting $m(S)=1$ for all the sublists $S\subseteq X$
of vectors in a matroid we get trivially a multiplicity function,
and hence a structure of an arithmetic matroid. In this case we
call $m$ \textit{trivial}. In fact this multiplicity function does
not add anything to the matroid structure. In this sense the
notion of an arithmetic matroid can be seen as a generalization of
the one of a matroid. Of course there are more interesting
examples.
\end{Rem}

\subsection{The main example}\label{ex:main}
The prototype of an arithmetic matroid (which in fact inspired our
definition) is the one that we are going to associate now to a
finite list $X$ of elements of a finitely generated abelian group
$G$. We recall that such a group is isomorphic to $G_f\oplus G_t$,
where $G_t$ is the torsion subgroup of $G$, which is finite, and
$G_f$ is free abelian, i.e it is isomorphic to $\mathbb{Z}^r$ for
some $r\geq 0$. Notice that in general we have many choices for
$G_f$, while $G_t$ is intrinsically defined.

Given a sublist $A\subseteq X$, we will denote by $\langle
A\rangle$ the subgroup of $G$ generated by the underlying set of
$A$.

We define the rank of a sublist $A\subseteq X$ as the maximal rank
of a free (abelian) subgroup of $\langle A\rangle$. This defines a
matroid structure on $X$.

For $A\in \mathbb{P}(X)$, let $G_A$ be the maximal subgroup of $G$
such that $\langle A\rangle\leq G_A$ and $|G_A: \langle
A\rangle|<\infty$, where $|G_A: \langle A\rangle|$ denotes the
index (as subgroup) of $\langle A\rangle$ in $G_A$. Then the
multiplicity $m(A)$ is defined as $m(A):=|G_A: \langle A\rangle|$.

Since we are interested in the multiplicities, clearly we can
always assume (and we will do it) that $\langle X\rangle$ has
finite index in $G$: otherwise we just replace $G$ by $G_X$, i.e.
the maximal subgroup of $G$ in which $\langle X\rangle$ has finite
index.

Notice that $m(\emptyset)$ equals the cardinality of $G_t$. In
particular $m(\emptyset)=1$ if and only if $G$ is free abelian.

We need to check that the function $m$ that we just defined is a
multiplicity function. The first axiom is easy to check. We
already observed that the second axiom for $m$ is just the first
axiom for the dual. Hence it would follow from the first one if we
can realize the dual arithmetic matroid as a list of elements in
a finitely generated group as we just did. This is the content of
the next section.

The third axiom for $m$ is proved in the following lemma.

\begin{Lem}
Given two lists $A$ and $B$ of elements of $G$ such that
$A\subseteq B$, and let $B$ be the disjoint union $B=A\cup T\cup
F$ such that for every $A\subseteq C\subseteq B$ we have
$rk(C)=rk(A)+|C\cap F|$. Then
$$
m(A)\cdot m(B)=m(A\cup T)\cdot m(A\cup F).
$$
\end{Lem}
\begin{proof}
We take the chance here to fix some notation and make some general
remark that we will use later in this work.

Recall that for any subgroup $H$ of $G$ we have $H= H_f\oplus
H_t$, where $H_t$ is the torsion subgroup of $H$, and $H_f$ is
free abelian. We will call $G_H$ the maximal subgroup of $G$ in
which $H$ has finite index. Notice that $G_H=(G_{H})_f\oplus G_t$.

\begin{Rem} \label{rm:quotient}
Let $G$ be a finitely generated abelian group, $H\leq G$ a
subgroup. With the notation above we have $G=G_f\oplus G_t$ and
$H=H_f\oplus H_t$, where necessarily $H_t\leq G_t$. Notice that,
since by the isomorphism theorem
$$
\frac{H+G_t}{G_t}\cong \frac{H}{H\cap G_t}=\frac{H}{H_t}\cong H_f,
$$
we can choose a suitable $H_f'\leq G_f$ such that
$H+G_t=H_f'\oplus G_t$, for which we must have
$$
H_f'\cong \frac{H_f'\oplus G_t}{G_t}=\frac{H+G_t}{G_t}\cong H_f.
$$
But again by the isomorphism theorem
$$
\frac{G}{H+G_t}\cong \frac{G/G_t}{(H+G_t)/G_t}=\frac{(G_f\oplus
G_t)/G_t}{(H_f'\oplus G_t)/G_t}\cong \frac{G_f}{H_f'}$$
and
$$\frac{H+G_t}{H}\cong \frac{G_t}{H\cap G_t}=\frac{G_t}{H_t},
$$
so
$$
|G:H|=|G:H+G_t|\cdot |H+G_t:H|=|G_f:H_f'|\cdot |G_t:H_t|.
$$
Therefore, as long as we are interested in the multiplicities,
eventually replacing $H=H_f\oplus H_t$ by $H':=H_f'\oplus H_t$, we
can always assume (and we will do it) that $H_f\subseteq G_f$.
\end{Rem}

Given a list $S\subseteq X$, we will write $G_S$ for $G_{\langle
S\rangle}$. Notice also that by definition, $rk(S)$ is the rank of
$\langle S\rangle_f$ (as free abelian group), and
$m(S)=|G_S:\langle S\rangle|$. Moreover, for $S\subseteq X$, we
let (cf. Remark \ref{rm:quotient}) $\langle S\rangle_f\subseteq
(G_S)_f$, so that we have
\begin{eqnarray*}
m(S) & = & |G_{S}:\langle S\rangle|=|(G_S)_f \oplus G_t:
\langle S\rangle_f \oplus \langle S\rangle_t|\\
 & = & |(G_S)_f:\langle S\rangle_f|\cdot |G_t: \langle S\rangle_t|.
\end{eqnarray*}

Observe that if $U\subseteq V\subseteq X$ and $rk(U)=rk(V)$, then
$G_U=G_V$.

By the isomorphism theorem we have
$$
\langle B\rangle/\langle A\cup F\rangle\cong \langle A\cup
T\rangle/(\langle A\cup T \rangle\cap \langle A\cup F\rangle).
$$
\begin{Claim}
It follows by our assumptions that $\langle A\cup T
\rangle\cap \langle A\cup F\rangle= \langle A\rangle$.
\end{Claim}
\begin{proof}[proof of the Claim]
We will prove the two inclusions.

Let $g\in \langle A\cup T \rangle\cap \langle A\cup F\rangle$, so
$$
g=\sum_{a\in A}\alpha_a a+\sum_{t\in T}\beta_t t= \sum_{a\in
A}\gamma_a a+\sum_{f\in F}\delta_f f
$$
where the $\alpha$'s, $\beta$'s, $\gamma$'s and $\delta$'s are
integers.

Let $F'\subseteq F$ be the subset of $F$ for which the
corresponding coefficients $\delta$'s are nonzero. If
$F'=\emptyset$, then $g=\sum_{a\in A}\gamma_a a\in \langle
A\rangle$.

If $F'\neq \emptyset$, then, letting $C:=A\cup F'$ we have by
assumption
$$
rk(C)= rk(A)+|F'|.
$$
But
$$
\sum_{f\in F'}\delta_f f=\sum_{f\in F}\delta_f f=\sum_{a\in
A}\alpha_a a+\sum_{t\in T}\beta_t t- \sum_{a\in A}\gamma_a a\in
\langle A\cup T\rangle
$$
and $rk(A\cup T)=rk(A)$ (just set $C:=A\cup T$), therefore
$$
rk(C)\leq rk(A)+|F'|-1,
$$
a contradiction.

The other inclusion is obvious.
\end{proof}

So we have
\begin{equation} \label{eq:lemex}
\langle B\rangle/\langle A\cup F\rangle\cong \langle A\cup
T\rangle/\langle A \rangle.
\end{equation}

Observing that $rk(B)=rk(A\cup F)$ and $rk(A\cup T)=rk(A)$, we
have
\begin{eqnarray*}
\frac{m(A\cup F)}{m(B)} & = & \frac{|G_{A\cup F}:\langle A\cup
F\rangle|}{|G_B:\langle B\rangle|} = \frac{|G_{B}:\langle A\cup
F\rangle|}{|G_B:\langle
B\rangle|}\\
 & = & |\langle
B\rangle:\langle
A\cup F\rangle|\\
 \text{(by \eqref{eq:lemex})} & = &
|\langle A\cup T\rangle:\langle
A\rangle|\\
 & = & \frac{|G_{A}:\langle A\rangle|}{|G_A:\langle
A\cup T\rangle|} = \frac{|G_{A}:\langle A\rangle|}{|G_{A\cup
T}:\langle A\cup
T\rangle|}\\
   & = & \frac{m(A)}{m(A\cup T)},
\end{eqnarray*}
as we wanted.

This completes the proof of the lemma.
\end{proof}

The fourth axiom for $m$ is a consequence of Lemma \ref{coco} (see
Section 4 for the relevant definitions): for $A\subseteq B$ and
$rk(A)=rk(B)$, we have that $\mu_B(A)$ equals the number of connected components of
$$H_A\setminus \bigcup_{B\supseteq T\supsetneq A}
H_T.
$$
Therefore it is clearly nonnegative.

The fifth axiom for $m$ is the dual of the fourth one, and again
it will follow from the realization of the dual arithmetic matroid
by a list of elements in a finitely generated abelian group,
which is proved in the next section.

\subsection{Representability}

We recall that a (classical) matroid is said to be
\textit{representable in characteristic 0} or
\emph{0-representable} if it is realized by a list of vectors in
$\mathbb{R}^n$.

We say that an arithmetic matroid is \emph{representable} if it is
realized by a list of elements in a finitely generated abelian
group.

By ``realized'' we mean that the rank and the multiplicity
functions are defined as in the Example \ref{ex:main}.

We say that an arithmetic matroid is:
\begin{itemize}
  \item \emph{0-representable} if its underlying matroid is;
  \item \emph{torsion-free} if $m(\emptyset)=1$;
  \item \emph{GCD} if its multiplicity function satisfies the
  \textit{GCD rule}:\\
  $m(A)$ equals the greatest common divisor (GCD) of the
  multiplicities of the maximal independent sublists of $A$, i.e.
  $$
  m(A):=GCD(\{m(B)\mid B\subseteq
  A\text{ and }|B|=rk(B)=rk(A)\}).
  $$
\end{itemize}
\begin{Rem}
If an arithmetic matroid is representable, then it is clearly
0-representable (just tensor with the rational numbers
$\mathbb{Q}$). Moreover, if it is also torsion-free, then it is
easily seen to be GCD (cf. Remark \ref{rm:GCD}).
\end{Rem}
This provides two classes of examples of non-representable arithmetic matroids:

\begin{Ex}
Let $\mathfrak{M}$ be non-0-representable. For example, consider
the Fano matroid, i.e. the matroid defined by the 7 nonzero
elements of $\mathbb{F}_2^3$, where $\mathbb{F}_2$ is the field
with two elements. Then every multiplicity function (e.g. the
trivial one) will make it into a non-representable arithmetic
matroid.
\end{Ex}

\begin{Ex}
Let us take $X=\{a,b,c\}$ and define $\mathfrak{M}$ as the matroid
on $X$ having bases $\{a,b\}$, $\{b,c\}$ and $\{a,c\}$ . Clearly,
$\mathfrak{M}$ is realized by three non-collinear vectors in a
plane. Now set the multiplicities of the bases to be $2$ and all
the others to be $1$. It is easy to check that this is an
arithmetic matroid, but it is not GCD, since $m(X)=1\neq 2$. Hence
it is not representable.
\end{Ex}

\section{Representability of the dual}

In this section we prove that the dual of a representable
arithmetic matroid is still representable. Our construction gives
an extension of the Gale duality \cite{Ga} to our setting.

Consider an arithmetic matroid $\mathfrak{M}$ represented by a
list $X$ of elements of a finitely generated abelian group $G$. We
are looking for a finitely generated abelian group $G'$ and a
finite list $X'$ of its elements representing the matroid
$\mathfrak{M}^*$.
\begin{Rem}
Notice that of course we need to have $|X|=|X'|$. Also, the rank
of $\mathfrak{M}^*$ must be $|X|$ minus the rank of
$\mathfrak{M}$.
\end{Rem}

We start with a presentation of our finitely generated abelian
group $G$ as $\mathbb{Z}^r\oplus (\mathbb{Z}/d_1\mathbb{Z})\oplus
(\mathbb{Z}/d_2\mathbb{Z})\oplus \cdots \oplus
(\mathbb{Z}/d_s\mathbb{Z})$, where $d_i$ divides $d_{i+1}$ for
$i=1,2,\dots ,s-1$. It is well known that such a presentation
exists and it is unique up to isomorphism. We realize this
presentation as a quotient $\mathbb{Z}^{r+s}/\langle Q\rangle$,
where $Q$ is the list of vectors $q_1,q_2,\dots ,q_s\in
\mathbb{Z}^{r+s}$, where $q_i$ has $d_i$ in the $(r+i)$-th
position, and $0$ elsewhere. We remember the order in which the
elements of the list $Q$ are given.

Now a finite list $X\subseteq G$ is given by a list of cosets
$X=\{\overline{v}_1,\overline{v}_2,\dots ,\overline{v}_k\}$, where
of course with $\overline{v}_i$ we denote the coset $v_i+\langle
Q\rangle$ for $v_i\in \mathbb{Z}^{r+s}$. We choose representatives
$v_i\in \mathbb{Z}^{r+s}$ for the cosets, which are
determined up to linear combinations of elements from $Q$. We set
$\widetilde{X}:=\{v_1,v_2,\dots,v_k\}$ a list of elements in
$\mathbb{Z}^{r+s}$. Also in this case we remember the order in
which the elements of $\widetilde{X}$ are given.

Hence we consider the $(r+s)\times (k+s)$ matrix
$[\widetilde{X}\sqcup Q]$, whose columns are the elements from
$\widetilde{X}$ in the given order first and from $Q$ in the given
order next. We call $(\widetilde{X}Q)^t$ the list of its rows in
the given order from top to bottom, which are vectors in
$\mathbb{Z}^{k+s}$. Hence we set $G':=\mathbb{Z}^{k+s}/\langle
(\widetilde{X}Q)^t\rangle$, and
$X':=\{\overline{e}_1,\overline{e}_2,\dots,\overline{e}_k\}$ will
be the list of cosets in $G'$, where as usual $e_i\in
\mathbb{Z}^{k+s}$ denotes the vector with $1$ in the $i$-th
position and $0$ elsewhere.

We call $(\mathfrak{M}',m')$ the arithmetic matroid associated to
the pair $(G',X')$.

We denote $\{1,2,\dots,k\}$ by $[k]$, and for $S\subseteq [k]$ we
denote by $S^c$ its complement in $[k]$, and we set
$\overline{v}_S:=\{\overline{v}_i\in X\mid i\in S\}$ and
$\overline{e}_S:=\{\overline{e}_i\in X'\mid i\in S\}$.

The main result of this section is the following theorem.
\begin{Thm} \label{thm:dual}
The bijection $\overline{e}_S\leftrightarrow \overline{v}_{S}$ for
$S\subseteq [k]$ is an isomorphism of arithmetic matroids between
$\mathfrak{M}'$ and $\mathfrak{M}^*$, i.e. it preserves both the
rank and the multiplicity functions.
\end{Thm}
\begin{proof}
We start with some easy observations. Take a sublist $A\subseteq
\widetilde{X}$ of elements of $\mathbb{Z}^{r+s}$, and set
$\overline{A}:=\{\overline{a}\mid a\in A\}\subseteq X$ a list of
elements of $G$. Then $\langle A\cup Q\rangle/\langle
Q\rangle\cong \langle \overline{A}\rangle$, where on the left we
are taking subgroups in $\mathbb{Z}^{r+s}$, while on the right we
are taking a subgroup of $G$. Hence the rank of $\overline{A}$
will be the same as the rank of $\langle A\cup Q\rangle$ minus the
rank of $\langle Q\rangle$.

Moreover, the multiplicity of $\overline{A}$ in $G$ is the same as
the multiplicity of $A\cup Q$ in $\mathbb{Z}^{r+s}$. In fact let
$\langle T\rangle/\langle Q\rangle$ be the maximal subgroup of $G$
in which $\langle \overline{A}\rangle$ has finite index, where
$T\subseteq \mathbb{Z}^{r+s}$. Then $$|\langle T\rangle/\langle
Q\rangle: \langle A\cup Q\rangle/\langle Q\rangle|=|\langle
T\rangle : \langle A\cup Q\rangle|$$ and $\langle T\rangle$ is
clearly the maximal subgroup of $\mathbb{Z}^{r+s}$ in which
$\langle A\cup Q\rangle$ has finite index.

Of course analogous observations apply to the sublists of $G'$,
with $(\widetilde{X}Q)^t$ in place of $Q$.

So to compute the ranks and the multiplicities of lists of
elements in $G$ or $G'$ (which is what we need to do if we want to
check that our map is an isomorphism of arithmetic matroid) we can
reduce ourself to compute them in $\mathbb{Z}^{r+s}$ or
$\mathbb{Z}^{k+s}$ respectively.
\begin{Rem} \label{rm:GCD}
Notice also that in $\mathbb{Z}^{m}$, to compute the multiplicity
of a list of elements, it is enough to see the elements as the
columns of a matrix, and to compute the greatest common divisor of
its minors of order the rank of the matrix (cf. \cite[Theorem
2.2]{StL0}).
\end{Rem}

We introduce a useful notation: given a list $Y$ of vectors in
$\mathbb{Z}^m$ given in some order, we denote by $[Y]$ the matrix
whose columns are the elements of $Y$ in the given order.

Given $S\subseteq [k]$, we want to compute the rank of
$\overline{e}_S$. Following our observations, first we want to
compute the rank of $[e_S\sqcup (\widetilde{X}Q)^t]$, where
$e_S:=\{e_i\mid i\in S\}$ and this is the matrix whose columns are
the elements of $e_S$ in some order first and the elements of
$(\widetilde{X}Q)^t$ in the given order next.

Notice that the matrix $[e_S\sqcup (\widetilde{X}Q)^t]$ looks like
$$
\left(%
\begin{array}{cc}
[e_S] & [\widetilde{X}]^t\\
\begin{array}{ccc}
0 & \cdots & 0   \\
 \vdots & & \vdots    \\
0 & \cdots & 0  \\
\end{array} &
[Q]^t
\end{array} \right)
$$
where the upper $t$ denotes the transpose of the matrix.

So the rank of $[e_S\sqcup (\widetilde{X}Q)^t]$ will be $|S|$
(looking at the $|S|$ rows indexed by $S$) plus the rank of
$v_{S^c}\cup Q$ (looking at the other rows), where
$v_{S^c}:=\{v_i\mid i\in S^c\}$. But as we already observed the
rank of $\overline{e}_{S}$ is the same as the rank that we just
computed minus the rank of $[(\widetilde{X}Q)^t]$, which is the
rank of $\widetilde{X}\cup Q$. Hence we just showed that the rank
of $\overline{e}_S$ is
$$
|S|-rk(\widetilde{X}\cup Q)+rk(v_{S^c}\cup Q
)=|S|-(rk(\widetilde{X}\cup Q)-rk(Q))+(rk(v_{S^c}\cup Q)-rk(Q)).
$$
As we already observed $rk(\widetilde{X}\cup Q)-rk(Q)$ and
$rk(v_{S^c}\cup Q)-rk(Q)$ are the ranks of $X$ and
$\overline{v}_{S^c}=X\setminus \overline{v}_S$ in the original
matroid, hence the rank of $\overline{e}_S$ is precisely the rank
of $\overline{v}_S$ in the dual.

Let us compute the multiplicity of $\overline{e}_S$. Using Remark
\ref{rm:GCD}, we have to compute the greatest common divisor of
the minors of maximal rank in the matrix $[e_S\sqcup
(\widetilde{X}Q)^t]$. Notice that any nonzero minor of maximal
order must involve all the rows indexed by $S$, otherwise we can
clearly get a nonzero minor of higher order using the missing
rows. But a nonzero minor of maximal order involving the rows
indexed by $S$ is clearly plus or minus a nonzero minor of maximal
order in the matrix $[v_{S^c}\cup Q]$. But those are exactly the
minors that we use to compute the multiplicity of
$\overline{v}_{S^c}$.

This proves that
$m'(\overline{e}_S)=m(\overline{v}_{S^c})=m^*(\overline{v}_S)$,
completing the proof of the theorem.
\end{proof}

This completes the proof that the ``main example'' (Section 1.4)
gives indeed an arithmetic matroid.

\section{Arithmetic Tutte polynomial and deletion-contraction}

\subsection{The classical Tutte polynomial}

The \emph{Tutte polynomial} $T_X(x,y)=T(\mathfrak{M}_X;x,y)$ of the matroid $\mathfrak{M}_X=(X,rk)$
is defined (in \cite{Tu}) as
$$T_X(x,y):= \sum_{A\subseteq X} (x-1)^{rk(X)-rk(A)} (y-1)^{|A|-rk(A)}.$$

From the definition it is clear that $T_X(1,1)$ is equal to the
number of bases of the matroid.

Although it is not clear from the definition, the coefficients of the Tutte polynomial are positive, and they
have a nice combinatorial interpretation. In fact, the Tutte
polynomial embodies two statistics on the list of the bases called
\textit{internal} and \textit{external activity}.

Let us fix a total order on $X$, and let $B$ be a basis extracted
from $X$.

We say that $v\in X\setminus B$ is \emph{externally active} on $B$
if $v$ is dependent on the list of elements of $B$ following it
(in the total order fixed on $X$). We say that $v\in B$ is
\emph{internally active} on $B$ if $v$ is externally active on the
complement $B^c:=X\setminus B$ in the dual matroid (where $B^c$ is
a basis).

The number $e(B)$ of externally active elements is called the
\emph{external activity} of $B$, while the number $i(B)=e^*(B^c)$
of internally active elements is called the \emph{internal
activity} of $B$.

The following result is proved in \cite{Cr}.
\begin{Thm}[Crapo]\label{Cra}
$$T_X(x,y)=\mathop{\sum_{B\subseteq X}}_{B \text{basis}}x^{e^*(B^c)}y^{e(B)}.$$
\end{Thm}

Hence the coefficients of $T_X(x,y)$ count the number of bases
having given internal and external activities.

\subsection{Arithmetic Tutte polynomial}

Following \cite{MoT}, we associate to an arithmetic matroid $\mathfrak{M}_X$ its
\textit{arithmetic Tutte polynomial}
$M_X(x,y)=M(\mathfrak{M}_X;x,y)$ defined as
\begin{equation} \label{eq:aritutte}
M_X(x,y):=\sum_{A\subseteq
X}m(A)(x-1)^{rk(X)-rk(A)}(y-1)^{|A|-rk(A)}.
\end{equation}

This polynomial has many applications. In particular it encodes
much combinatorial information on toric arrangements (as we will
recall in Section 4), on zonotopes (\cite{DM}), and on Dahmen-Micchelli spaces
(\cite{MoT}, \cite{DMi}, \cite{DM2}; see also \cite{li}, \cite{HRX}, \cite{Le}
for related topics).

It has been shown in \cite{MoT} that if $\mathfrak{M}_X$ is
representable, the coefficients of this polynomial are positive.

Our main goal will be to give a combinatorial interpretation of
this polynomial for \textit{any} arithmetic matroid. By doing
this, we will also extend the positivity result.

But before that, we want to discuss briefly the connection between
the axioms that we gave for an arithmetic matroid and this
polynomial.
\begin{Rem} \label{rm:axioms}
We make some remark on the independence of the axioms of a
multiplicity function.

Consider the matroid on $X_1=\{t\}$, where $rk(X_1)=0$. Setting
$m(X_1)=2$ and $m(\emptyset)=3$, we have that $m$ satisfies all
the axioms of a multiplicity function, except the first one.

Consider the matroid on $X_2=\{f\}$, where $rk(X_2)=1$. Setting
$m(X_2)=3$ and $m(\emptyset)=2$, we have that $m$ satisfies all
the axioms of a multiplicity function, except the second one.

The next two examples were suggested to us by Petter Br\"{a}nd\'{e}n.

Consider the matroid on $X_3=\{f,t\}$, where
$rk(\{f\})=rk(\{f,t\})=1$ and $rk(\{t\})=rk(\emptyset)=0$, i.e.
$\{f\}$ is the only basis. Setting
$m(\{f\})=m(\{f,t\})=m(\emptyset)=2$ and $m(\{t\})=1$ we have that
$m$ satisfies all the axioms of a multiplicity function, except
the third one. Moreover, applying formula \eqref{eq:aritutte} we
would get $M_{X_3}(x,y)=x+y+xy-1$. 

Consider the matroid on $X_4=\{t_1,t_2,t_3,t_4\}$, where
$rk(X_4)=0$. Setting $m(\emptyset)=4$, $m(\{t_i\})=2$ for
$i=1,2,3,4$, and $m$ equal one for all the other sublists, we have
that $m$ satisfies all the axioms of a multiplicity function,
except the fourth one. Moreover, applying formula
\eqref{eq:aritutte} we would get $M_{X_4}(x,y)=y^4+4y-1$.

Consider the matroid on $X_5=\{f_1,f_2,f_3,f_4\}$, where
$rk(A)=|A|$ for all $A\subseteq X$, i.e. $X_5$ is the only basis.
Setting $m(X_5)=4$, $m(A)=2$ for $A\subseteq X$ and $|A|=3$, and
$m$ equal one for all the other sublists, we have that $m$
satisfies all the axioms of a multiplicity function, except the
fifth one. Moreover, applying formula \eqref{eq:aritutte} we would
get $M_{X_5}(x,y)=x^4+4x-1$.

Summarizing, each of the axioms of a multiplicity function is
independent on the other ones. Moreover, even dropping only axiom
(3) or only axiom (4) or only axiom (5) we can get an arithmetic
Tutte polynomial with negative coefficients. In this sense,
without those axioms we would get a ``non-combinatorial'' object.
\end{Rem}

\subsection{Deletion and contraction}

We introduce two fundamental constructions.
Recall that given a matroid $\mathfrak{M}_X$ and a vector $v\in
X$, we can define the \textit{deletion} of $\mathfrak{M}_X$ as the
matroid $\mathfrak{M}_{X_1}$, whose list of vectors is
$X_1:=X\setminus \{v\}$, and whose independent lists are just the
independent lists of $\mathfrak{M}_X$ contained in $X_1$. Notice
that the rank function $rk_1$ of $\mathfrak{M}_{X_1}$ is just the
restriction of the rank function $rk$ of $\mathfrak{M}_{X}$.

Given an arithmetic matroid $(\mathfrak{M}_X,m)$ and a vector
$v\in X$, we define the \textit{deletion} of $(\mathfrak{M}_X,m)$
as the arithmetic matroid $(\mathfrak{M}_{X_1},m_1)$, where
$\mathfrak{M}_{X_1}$ is the deletion of $\mathfrak{M}_{X}$ and
$m_1(A):=m(A)$ for all $A\subseteq X_1=X\setminus \{v\}$. It is
easy to check that this is in fact an arithmetic matroid.

Recall that given a matroid $\mathfrak{M}_X$ and a vector $v\in
X$, we can define the \textit{contraction} of $\mathfrak{M}_X$ as
the matroid $\mathfrak{M}_{X_2}$, whose list of vectors is
$X_2:=X\setminus \{v\}$, and whose rank function $rk_2$ is given
by $rk_2(A):=rk(A\cup \{v\})-rk(\{v\})$, where of course $rk$ is
the rank function of $\mathfrak{M}_X$.

Given an arithmetic matroid $(\mathfrak{M}_X,m)$ and a vector
$v\in X$, we define the \textit{contraction} of
$(\mathfrak{M}_X,m)$ as the arithmetic matroid
$(\mathfrak{M}_{X_2},m_2)$, where $\mathfrak{M}_{X_2}$ is the
contraction of $\mathfrak{M}_{X}$ and $m_2(A):=m(A\cup \{v\})$ for
all $A\subseteq X_2=X\setminus \{v\}$. It is easy to check that
this is in fact an arithmetic matroid.

\begin{Ex}
If an arithmetic matroid $(\mathfrak{M}_X,m)$ is represented by a
list $X$ of elements of $G$, it is easy to check that the deletion
corresponds to the arithmetic matroid $(\mathfrak{M}_{X_1},m_1)$
of the sublist $X_1:=X\setminus \{v\}$, while the contraction
corresponds to the arithmetic matroid $(\mathfrak{M}_{X_2},m_2)$
of the list $\overline{X}:=\{a+\langle v\rangle\mid a\in
X\setminus \{v\}\}$ of cosets in $G/\langle v\rangle$.
\end{Ex}

Observe that the deletion of $v\in X$ in $\mathfrak{M}_X$
corresponds to the contraction of $v\in X$ in
$\mathfrak{M}_{X}^*$, and viceversa the contraction of $v\in X$ in
$\mathfrak{M}_X$ corresponds to the deletion of $v\in X$ in
$\mathfrak{M}_{X}^*$.

\subsection{Free, torsion, and proper vectors}
Given an element $v\in X$, we denote by $rk_1$ and $rk_2$ the rank
function of the deletion and the contraction by $v$ respectively.

We say that $v\in X$ is:
\begin{itemize}
  \item \textit{free} if both
$rk_1(X\setminus\{v\})=rk(X\setminus\{v\})=rk(X)-1$ and
$rk_2(X\setminus \{v\})=rk(X)-1$;
  \item \textit{torsion} if both
$rk_1(X\setminus\{v\})=rk(X)$ and $rk_2(X\setminus \{v\})=rk(X)$;
  \item  \textit{proper} if both
$rk_1(X\setminus\{v\})=rk(X)$ and $rk_2(X\setminus
\{v\})=rk(X)-1$.
\end{itemize}

Observe that any vector of a matroid is of one and only one of the
previous three types.

\begin{Ex}
If we look at the arithmetic matroid represented by a list $X$ of
elements of $G$, the torsion vectors are precisely the torsion
elements in the algebraic sense, while a free vector will be an
element of $G$ which is not torsion and such that the sum $\langle
X\setminus \{v\}\rangle \oplus \langle v\rangle$ is direct.
\end{Ex}

\begin{Rem}\label{FTP}
A vector is free in a matroid if and only if it is torsion in its
dual. While a vector is proper in a matroid if and only if it is
proper in its dual.

Moreover, suppose that $v$ and $w$ are two distinct vectors, and
we make a deletion with respect to $w$. If $v$ is free or torsion,
then it is again free or torsion respectively in the deletion
matroid. While if $v$ is proper, then it can be proper or free,
but not torsion in the deletion matroid.

Dually, if we make a contraction with respect to $w$, then if $v$
is free or torsion, then it is again free or torsion respectively
in the contraction matroid. While if $v$ is proper, then it can be
proper or torsion, but not free in the contraction matroid.
\end{Rem}

\subsection{Molecules}\label{s-mole}
We define a \emph{molecule} as an arithmetic matroid that does not
have proper vectors.

Hence a molecule will be given by a list of the form
$X=\{f_1,f_2,\dots,f_r,t_1,t_2,\dots,t_s\}$, where the $f_i$'s are
free vectors, and the $t_j$'s are torsion vectors.

Notice that by Remark \ref{FTP} the dual of a molecule is still a molecule.

\begin{Rem} \label{rm:axiom3}
Notice that in the assumption of axiom (3) of a multiplicity
function, we are simply asking that if we do the deletion of the
elements of $X\setminus B$ and the contraction of the elements of
$A$ we are left with a molecule, whose only basis is going to be
$F$.
\end{Rem}

Looking at the underlying matroid, a molecule consists of a
(unique) basis plus a bunch of rank $0$ elements. For example, in
the $0$-representable case, the latter ones would just correspond
to a bunch of zeros.

The classical Tutte polynomial of such a matroid turns out to be
very simple: for our $X$ it would correspond to the monomial $x^ry^s$,
where $r$ is the rank of the matroid, and $s$ is the number of
rank $0$ elements.

In fact, by deletion-contraction, one can reduce the computation
of the classical Tutte polynomial (but also the proof of several
properties of a matroid) to the singletons (which are necessarily
torsion or free), which are usually called the \emph{atoms} of the
matroid (this justifies our \textit{molecules}).

But we will see that the arithmetic Tutte polynomial of a molecule
is not so simple. In fact, we will give our combinatorial
interpretation first in the case of molecules. Then we will extend
it to the general case. Indeed, to prove the general case, we will
apply recursively deletion-contraction for all the proper vectors,
reducing ourselves to the molecules.

\subsection{Direct sum}

Given two matroids $\mathfrak{M}_{X_1}=(X_1,I_1)$ and
$\mathfrak{M}_{X_2}=(X_2,I_2)$, we can form their \textit{direct
sum}: this will be the matroid
$\mathfrak{M}_X=\mathfrak{M}_{X_1}\oplus \mathfrak{M}_{X_2}$ whose
list of vectors is the disjoint union $X:=X_1\sqcup X_2$, and
where the independent lists will be the disjoint unions of lists
from $I_1$ with lists from $I_2$. Hence for any sublist
$A\subseteq X$, the rank $rk(A)$ of $A$ will be the sum of the
rank $rk_1(A\cap X_1)$ of $A\cap X_1$ in $\mathfrak{M}_{X_1}$ with
the rank $rk_2(A\cap X_2)$ of $A\cap X_2$ in $\mathfrak{M}_{X_2}$.

If the two matroids are $0$-representable in two vector spaces
$V_1$ and $V_2$ respectively, the direct sum matroid corresponds
of course to the matroid of the list $X:=X_1\sqcup X_2$ in the
direct sum $V_1\oplus V_2$, with the obvious identification of the
the subspaces $V_1$ and $V_2$.

It follows immediately from the definition of the Tutte polynomial
that in this case
$$
T_X(x,y)=T_{X_1}(x,y)\cdot T_{X_2}(x,y).
$$

Given two arithmetic matroids $(\mathfrak{M}_{X_1},m_1)$ and
$(\mathfrak{M}_{X_2},m_2)$ we define their \textit{direct sum} as
the arithmetic matroid $(\mathfrak{M}_{X},m)$, where
$\mathfrak{M}_{X}:=\mathfrak{M}_{X_1}\oplus \mathfrak{M}_{X_2}$,
and for any sublist $A\subseteq X=X_1\sqcup X_2$, we set
$m(A):=m_1(A\cap X_1)\cdot m_2(A\cap X_2)$. It is easy to check
that this is indeed an arithmetic matroid.

Again, it is clear from the definition of the arithmetic Tutte
polynomial that in this case
$$
M_X(x,y)=M_{X_1}(x,y)\cdot M_{X_2}(x,y).
$$

If the two arithmetic matroids are represented by a list $X_1$ of
elements of a group $G_1$ and a list $X_2$ of elements of a group
$G_2$, then, with the obvious identifications, $X:=X_1\sqcup X_2$
is a list of elements of the group $G:=G_1\oplus G_2$, and the
arithmetic matroid associated to this list is exactly the direct
sum of the two.
\begin{Ex} \label{ex:molecule}
Consider a molecule given by a list
$X=\{f_1,f_2,\dots,f_r,t_1,t_2,\dots,t_s\}$ of elements of a group
$G=G_f\oplus G_t$, where the $f_i$'s are free and the $t_j$'s are
torsion. In this case, up to changing some $f_i$ by adding some
element of $G_t$ (cf. Remark \ref{rm:quotient}), we can assume
$\{f_1,f_2,\dots,f_r\}\subseteq G_f$. Then we can regard this as a
direct sum of the arithmetic matroid associated to the list
$X_f:=\{f_1,f_2,\dots,f_r\}$ of elements of $G_f$ and
$X_t:=\{t_1,t_2,\dots,t_s\}$ of elements of $G_t$.

Hence in this case
$$
M_X(x,y)=M_{X_f}(x,y)\cdot M_{X_t}(x,y).
$$
\end{Ex}

\section{Geometry of the generalized toric arrangement}

The aim of this section is to explain the geometrical ideas
underlying the combinatorial concepts studied in this paper, and
motivating them. A reader only interested in the combinatorics may
skip this section without affecting the comprehension of what
follows.

\subsection{Generalized tori}

Let $G=G_f\oplus G_t$ be a finitely generated abelian group, where
$G_t$ denotes the torsion subgroup of $G$, and $G_f$ is some free
abelian group, and define
$$T(G):= {Hom(G,\mathbb{C}^*)}.$$
$T(G)$ has a natural structure of abelian linear algebraic group.
In fact it is the direct sum of a complex torus $T(G_f)$ (whose
dimension is the rank as free abelian group of $G_f$) and of the
finite group $T({G_t})$ dual to ${G_t}$ (and isomorphic to it).
Topologically, this is the disjoint union of $|G_t|$ copies of the
torus $T(G_f)$.

Moreover, $G$ is identified with $Hom(T(G),\mathbb{C}^*)$, the
group of characters of $T(G)$: indeed given $\lambda\in G$ and
$t\in T(G)=Hom(G,\mathbb{C}^*)$ we set
$$
\lambda(t):= t(\lambda).
$$

In the same way, we can define
$$T_{\mathbb{R}}(G):= {Hom(G,\mathbb{S}^1)}$$
where we set $\mathbb{S}^1:= \{z\in \mathbb{C}\:|\: |z|=1\}.$ Then
$T_{\mathbb{R}}(G)$ has a natural structure of abelian compact
real Lie group, having $G$ as its group of characters. Again, $G$
is identified with $Hom(T(G),\mathbb{S}^1)$. In fact the functor
$Hom(\:\cdot\:,\mathbb{S}^1)$ gives rise to the so-called
\emph{Potryagin duality}.

When it is not ambiguous, we will denote $T(G)$ by $T$ and
$T_{\mathbb{R}}(G)$ by $T_\mathbb{R}$.

\subsection{Generalized toric arrangements}

Let $X\subseteq G$ be a finite list, spanning a finite index
subgroup of $G$. The kernel of every character $\lambda\in X$ is a subvariety in $T(G)$:
$$
H_\lambda:= \big\{ t\in T \mid \lambda(t)=1 \big\}.
$$
More precisely, $H_\lambda$ is the union of a bunch of connected components of
$T(G)$ if the rank of $\{\lambda\}$ is zero, and a (not
necessarily connected) hypersurface of $T(G)$ if the rank of
$\{\lambda\}$ is one.

The collection $\mathcal{T}(X)=\left\{H_\lambda\mid \lambda\in
X\right\}$ is called the \emph{generalized toric arrangement}
defined by $X$ on $T$.

We denote by $\rx$ the complement of the arrangement in $T$:
$$
\rx:= T\setminus\bigcup_{\lambda\in X} H_\lambda.
$$
We also denote by $\mathcal{C}(X)$ the set of all the connected
components of all the intersections of the subvarieties
$H_\lambda$, ordered by reverse inclusion and having as minimal
elements the connected components of $T$.

Of course we will have similar definitions for $T_\mathbb{R}$. We
will denote with a subscript ``$\mathbb{R}$'' these real
counterparts (e.g. $\rx_{\mathbb{R}}$).

In particular, when $G$ is free, $T$ is a torus and
$\mathcal{T}(X)$ is called the \emph{toric arrangement} defined by
$X$. Such arrangements have been studied for instance in
\cite{DPt}, \cite{Mo}, \cite{ERS}. In particular, the complement $\rx$ has been described
topologically in \cite{MS}, \cite{DD} and geometrically in \cite{Mw}. In this description the poset
$\mathcal{C}(X)$ plays a major role, analogous to that of the
intersection poset for hyperplane arrangements (see \cite{DPt},
\cite{Mw})

\subsection{Relations with the arithmetic Tutte polynomial}

In this subsection we recall some facts, which were proved in
\cite{MoT}.

Given $A\subseteq X$ let us define
$$H_A:= \bigcap_{\lambda \in A}H_\lambda.$$

The following fact is a simple consequence of Pontryagin duality:

\begin{Lem}\label{coco}
$m(A)$ is equal to the number of connected components of $H_A$.
\end{Lem}

Then the arithmetic Tutte polynomial is deeply related with
generalized toric arrangements, and in fact it was introduced to
study them. We recall some results from \cite{MoT}.

\begin{Thm} \label{thm:Luca}
\begin{enumerate}
  \item The number of connected components of $\rx_\mathbb{R}$ is $M_X(1,0)$.
  \item the Poincar\'{e} polynomial of $\rx$ is $q^n M_X\left(\frac{2q+1}{q},0\right)$.
  \item the characteristic polynomial of $\mathcal{C}(X)$ is $(-1)^n M_X(1-q,0)$.
\end{enumerate}
\end{Thm}

Since $rank(G_f)= dim(T)$, the maximal (in the reverse order!)
elements of $\mathcal{C}(X)$ are 0-dimensional, hence (since they
are connected) they are points. We denote by $\mathcal{C}_0(X)$
the set of such elements, which we call the \emph{points} of the
arrangement. For every $p\in\mathcal{C}_0(X)$, let us define
$$X_p:= \left\{\lambda\in X | p\in H_\lambda\right\}.$$

Then we have:

\begin{Lem}\label{aes}
$$M_X(1,y)=\sum_{p\in \mathcal{C}_0(X)} T_{X_p}(1,y).$$
\end{Lem}

This lemma will be the starting point of our combinatorial
interpretation, as we will explain in the next section.

\section{Towards the combinatorial interpretation}

\subsection{General considerations}

We want to give a combinatorial interpretation of the arithmetic
Tutte polynomial. Let us look at a very easy but already
nontrivial example.
\begin{Ex} \label{ex:1}
Consider the list $$X=\{v_1:=(1,1),v_2:=(1,-1)\}\subseteq
\mathbb{Z}^2.$$ In this case we only have two free vectors, which
together form the only basis of the matroid $\mathfrak{M}_X$. Here
the multiplicity function is given by $m(X)=2$ and
$m(\emptyset)=m(\{v_1\})=m(\{v_2\})=1$.

If we compute the polynomial $M_X(x,y)$ using the defining formula
we get
$$
M_X(x,y)=x^2+1.
$$

Notice that the dual matroid $\mathfrak{M}_{X}^*$ is a rank $0$
matroid, whose only basis is the empty list. Moreover in the dual
arithmetic matroid $(\mathfrak{M}_{X}^*,m^*)$ the empty list has
multiplicity $2$.
\end{Ex}

We start our analysis with some general considerations. First of
all we observe that specializing at $x=1$ and $y=1$ the polynomial
$M_X(x,y)$ we get the sum of the multiplicities of the bases of
the matroid. Notice also that the bases in the matroid
$\mathfrak{M}_X$ correspond bijectively with the bases of
$\mathfrak{M}_{X}^*$ under the involution of complementing with
respect to $X$, and in fact by definition the multiplicity of a
basis in $(\mathfrak{M}_X,m)$ is the same as the multiplicity of
the complement in the dual $(\mathfrak{M}_{X}^*,m^*)$.

Hence, keeping in mind what Crapo did with the Tutte polynomial,
it is natural to try to interpret the polynomial
$\mathfrak{M}_X(x,y)$ as a sum over the bases counted with
multiplicity of monomials in $y$ and $x$, whose exponents give
some statistics on the basis and its complement (in the dual)
respectively.

Already in the very simple Example \ref{ex:1} we see one of the
difficulties of our task: in this example we only have one basis
counted with multiplicity $2$ both in $(\mathfrak{M}_{X},m)$ and
in $(\mathfrak{M}_{X}^*,m^*)$, but the two monomials are distinct!

The problem here is that the monomials of $M_X(x,y)$ are counted
by a list and not just a set. Moreover, apparently for identical
elements of the list the statistic may differ.

It turns out that a key ingredient for the understanding of the
combinatorics behind the polynomial $M_X(x,y)$ is a suitable list
of maximal rank sublists of $X$. The geometric considerations
exposed in the previous section suggested us to look at them in
the first place.

\subsection{Two fundamental lists}

Starting with our arithmetic matroid $(\mathfrak{M}_X,m)$, we
construct a list $L_X$ of maximal rank sublists of $X$ in the
following way.

To every maximal rank sublist $S$ of $X$ we associate the
nonnegative (axiom (4)) integer
$$
\mu(S):=\sum_{T\supseteq S}(-1)^{|T|-|S|}m(T).
$$

Then the list $L_X$ is defined as the list in which each maximal
rank sublist $S$ appears $\mu(S)$ many times.

Notice that if we extract the bases from our list $L_X$, each
basis $B$ will show up exactly $m(B)$ times: in fact, by
inclusion-exclusion, each basis $B$ will appear $\sum\mu(T)=m(B)$
times, where the sum is taken over the sublists $T$ that contain
$B$.

Dually, we construct the list $L_{X}^*$ in the same way from the
dual arithmetic matroid $(\mathfrak{M}_{X}^*,m^*)$.

\begin{Ex}
We consider again the Example \ref{ex:1}. In this case we have
$L_X=(X,X)$, while $L_{X}^*=(X,\emptyset)$.
\end{Ex}

We introduce the following notation. The list of pairs $(B,T)$,
where $B$ is a basis, $B\subseteq T$ and $T\in L_X$, counted with
multiplicity $\mu(T)$, will be denoted by $\mathcal{B}$. The
corresponding list in the dual will be denoted by $\mathcal{B}^*$.

\subsection{Local external activity}

We already observed that the multiplicity of the basis $B$ in
$\mathfrak{M}_{X}$ is the same as the multiplicity of the basis
$B^c$ in $\mathfrak{M}_{X}^*$. So it is now natural to interpret
the polynomial $M_X(x,y)$ as a sum over the elements of
$\mathcal{B}$ of monomials in $x$ and $y$:

For every such pair $(B,T)$ we define the statistic $e(B,T)$ to be
the \textit{local external activity} of the basis $B$ \textit{in}
the list $T$, i.e. the number of elements of $T\setminus B$ that
are externally active on $B$. Notice that the torsion elements of
$T$ are always active (if you don't want to deal with the empty
list, this is a convention). Dually, we define
$e^*(B^c,\widetilde{T})$ in the same way for the basis $B^c$ in
the dual and $B^c\subseteq \widetilde{T}\in L_{X}^*$.

More explicitly we would like to see $M_X(x,y)$ as $\sum
x^{e^*(B^c,\widetilde{T})}y^{e(B,T)}$.

\begin{Ex}
We consider again the Example \ref{ex:1}. In this case we have two
identical pairs $(X,X)$ in the original arithmetic matroid, where
obviously $e(X,X)=0$, while in the dual we have two distinct pairs
$(\emptyset,X)$ and $(\emptyset,\emptyset)$, where
$e^*(\emptyset,X)=2$ and $e^*(\emptyset,\emptyset)=0$.

In fact in this case the polynomial $M_X(x,y)$ is $x^2+1$.
\end{Ex}

\begin{Rem}
This definition of local external activity is motivated by Lemma \ref{aes}.

Indeed, this lemma tells us that the exponents of $y$ are the
external activities of the bases computed in the lists $X_p$,
hence they are the local external activities $e(B,T)$. Therefore,
at least for a representable arithmetic matroid, we have:

$$M_X(1,y)=\sum_{(B,T)\in \mathcal{B}} y^{e(B,T)}.$$

Furthermore, since the dual of a representable matroid is still
representable, we have a dual toric arrangement. The same
considerations then allow to conclude that
$$M_X(x,1)=\sum_{(B^c,\widetilde{T})\in \mathcal{B}^*} x^{e^*(B^c,\widetilde{T})}.$$
\end{Rem}

In order to conclude our construction we need to face a nontrivial
problem.

\subsection{The matching problem}

The problem here is again that we have a list of pairs and not
just a set. So a pair $(B,T)$ can appear several times, as we have
seen in the example above, and it needs to be matched with a
suitable pair $(B^c,\widetilde{T})$. In the last example we didn't
have the problem of the matching since the statistic for the $y$
was always $0$. But in general there could be many choices.

In general we have the problem of matching a pair $(B,T)$ with a
suitable pair $(B^c,\widetilde{T})$.

\section{The molecular case}

In this section we consider the special case of a molecule, i.e.
of an arithmetic matroid $\mathfrak{M}$ in which there are no
proper vectors (see Section \ref{s-mole}). Hence $\mathfrak{M}$ is given by a list $X$ of the
form
$$X=\{f_1,f_2,\dots,f_r,t_1,t_2,\dots,t_s\}$$ where the $f_i$'s are
free vectors, and the $t_j$'s are torsion vectors.

We want to find a combinatorial interpretation of the polynomial
$M_X(x,y)$ in this case. As we have seen in the previous section,
we have the problem of matching the pairs $(B,T)\in \mathcal{B}$,
where $B$ is a basis, $B\subseteq T$ and $T\in L_X$, with the
pairs $(B^c,\widetilde{T})\in \mathcal{B}^*$, where $B^c$ is of
course a basis in the dual, $B^c\subseteq \widetilde{T}$ and
$\widetilde{T}\in L_{X}^*$.

In the special case that we are considering the matching will be
done in the following way.

The idea is that we want to match the copies of a pair $(B,T)\in
\mathcal{B}$ \textit{evenly} among the copies of pairs
$(B^c,\widetilde{T})\in \mathcal{B}^*$, and viceversa. With this
we mean the following. Let $\ell((B,T),(B^c,\widetilde{T}))$ be
the number of copies of $(B,T)$ that we match with copies of
$(B^c,\widetilde{T})$. Than for distinct $T_1$ and $T_2$ in $L_X$
we want that
$$
\frac{\mu(T_1)}{\mu(T_2)}=\frac{\ell((B,T_1),(B^c,\widetilde{T}))}{\ell((B,T_2),(B^c,\widetilde{T}))}
$$
for every $\widetilde{T}\in L_{X}^*$. Dually, we want also
$$
\frac{\ell((B,T),(B^c,\widetilde{T}_1))}{\ell((B,T),(B^c,\widetilde{T}_2))}=\frac{
\mu^*(\widetilde{T}_1)}{\mu^*(\widetilde{T}_2)}
$$
for distinct $\widetilde{T}_1$ and $\widetilde{T}_2$ in $L_{X}^*$,
and for every $T\in L_X$.

We call this property \textit{equidistribution} of the matching.
Before showing that this is in fact possible, we assume that we
can do that, and we make some remarks.

Let us call $\psi=\psi_B$ such a bijection between $\mathcal{B}$
and $\mathcal{B}^*$. Then notice that this bijection is in fact
unique up to identification of the copies of the same pairs.

Let us set
$$
\overline{M}_X(x,y):=\sum_{(B^c,\widetilde{T})\in
\mathcal{B}^*}x^{e^*(B^c,\widetilde{T})}y^{e(\psi^{-1}(B^c,\widetilde{T}))}.
$$

Now notice that the equidistribution property implies that this
polynomial is in fact a product of two polynomials, one in $x$ and
one in $y$. In fact we have that $\overline{M}_B(x,y)\cdot
\overline{M}_{B^c}(x,y)$ is an integer multiple of
$\overline{M}_X(x,y)$, where $\overline{M}_B(x,y)$ is a polynomial
in $x$, since clearly there is no external activity on $B$, and
$\overline{M}_{B^c}(x,y)$ is a polynomial in $y$, since clearly
there is no external activity in the dual (cf. Example
\ref{ex:basetorsion}).

\begin{Rem}
Suppose that our arithmetic matroid is representable in a group
$G=G_f\oplus G_t$. If we denote by $M_{X_f}(x,y)$ the arithmetic
Tutte polynomial of the list $X_f:=\{f_1,f_2,\dots, f_r\}$ of
elements of $G_f$, and by $M_{X_t}(x,y)$ the arithmetic Tutte
polynomial of the list $X_t:=\{t_1,t_2,\dots,t_s\}$ of elements of
$G_t$, then we already observed in Example \ref{ex:molecule} that
$$
M_{X}(x,y)=M_{X_f}(x,y) \cdot M_{X_t}(x,y).
$$
In fact, in this case $M_B(x,y)=m(\emptyset)\cdot M_{X_f}(x,y)$
while $M_{B^c}(x,y)=M_{X_f}(x,y)$, so that $M_B(x,y)\cdot
M_{B^c}(x,y)=m(\emptyset)\cdot M_X(x,y)$ (cf. Example
\ref{ex:basetorsion}).
\end{Rem}

We can now give our combinatorial interpretation of the arithmetic
Tutte polynomial in the molecular case.
\begin{Thm} \label{thm:specialcase}
If $(\mathfrak{M}_X,m)$ is an arithmetic matroid with no proper
vectors, then
$$
M_X(x,y)=\overline{M}_X(x,y)=\sum_{(B^c,\widetilde{T})\in
\mathcal{B}^*}x^{e^*(B^c,\widetilde{T})}y^{e(\psi^{-1}(B^c,\widetilde{T}))}.
$$
\end{Thm}

Before proving this theorem, we show an example.

\begin{Ex} \label{ex:basetorsion}
Let
$$X=\{a:=(1,2,\overline{0}),b:=(2,0,\overline{1}),c:=(0,0,\overline{2}),
d:=(0,0,\overline{3})\}\subseteq G:=\mathbb{Z}^2\oplus
\mathbb{Z}/6\mathbb{Z}.$$ We have no proper vectors, so the only
basis is $B=\{a,b\}$, while $c$ and $d$ are two torsion vectors.

A straightforward computation shows that
\begin{eqnarray*}
M_X(x,y) & = & 4+6y+2x+2x^2+3x^2y+3xy+2y^2+x^2y^2+xy^2\\
 & = & (y+2)(y+1)(x^2+x+2).
\end{eqnarray*}
Observe also that $M_{\{a,b\}}(x,y)=6(x^2+x+2)$, while
$M_{\{c,d\}}(x,y)=(y+2)(y+1)$, so their product is a multiple of
(in fact $m(\emptyset)=6$ times) $M_X(x,y)$ as it should be.

To compute the multiplicities we look at the matrix
$$
\left(%
\begin{array}{ccc}
  1 & 2 & 0 \\
  2 & 0 & 1 \\
  0 & 0 & 2 \\
  0 & 0 & 3 \\
  0 & 0 & 6 \\
\end{array}%
\right)
$$
whose rows are representatives of the elements of $X$ in
$\mathbb{Z}^3$ together with the vector $q:=(0,0,6)$, where we
think of $G=\mathbb{Z}^2\oplus \mathbb{Z}/6\mathbb{Z}$ as
$\mathbb{Z}^3/\langle q\rangle$. Following Remark \ref{rm:GCD}, we
can compute the multiplicities of $A\subseteq X$ by looking at the
greatest common divisor of the nonzero minors of maximal rank that
we can extract from the corresponding rows of our matrix, together
with the last row $q$.

We have $m(\emptyset)=6$, $m(\{a\})=6$, $m(\{b\})=12$,
$m(\{c\})=2$, $m(\{d\})=3$, $m(\{a,b\})=24$, $m(\{a,c\})=2$,
$m(\{a,d\})=3$, $m(\{b,c\})=4$, $m(\{b,d\})=6$, $m(\{c,d\})=1$,
$m(\{a,b,c\})=8$, $m(\{a,b,d\})=12$, $m(\{a,c,d\})=1$,
$m(\{b,c,d\})=2$, $m(X)=4$.

To construct the list $L_X$, we look at the maximal rank sublists
of $X$, and we compute their $\mu$'s. So first take
$\mu(X)=m(X)=4$ copies of $X$. Then we take
$\mu(\{a,b,c\})=m(\{a,b,c\})-\mu(X)=8-4=4$ copies of $\{a,b,c\}$,
$\mu(\{a,b,d\})=m(\{a,b,d\})-\mu(X)=12-4=8$ copies of $\{a,b,d\}$,
and finally
$\mu(\{a,b\})=m(\{a,b\})-(\mu(X)+\mu(\{a,b,c\})+\mu(\{a,b,d\}))=24-16=8$
copies of $\{a,b\}$. Hence, using an exponential notation for the
number of copies of an element in a list, we have
$L_X=(X^4,\{a,b,c\}^4,\{a,b,d\}^{8},\{a,b\}^8)$.

Doing the same procedure in the dual, we get
$L_{X}^*=(X^6,\{a,c,d\}^6,\{c,d\}^{12})$.

Let us compute the external activities of the pairs of
$\mathcal{B}$ and $\mathcal{B}^*$. Here it is extremely easy,
since all the elements that are not in the basis are active. Hence
$e(\{a,b\},X)=2$, $e(\{a,b\},\{a,b,c\})=e(\{a,b\},\{a,b,d\})=1$
and $e(\{a,b\},\{a,b\})=0$, while $e^*(\{c,d\},X)=2$,
$e^*(\{c,d\},\{a,c,d\})=1$ and finally $e^*(\{c,d\},\{c,d\})=0$.

Now the bijection $\psi$: we have to equidistribute the pairs in
$\mathcal{B}$ with the pairs in $\mathcal{B}^*$. Let us call
$\ell((B,T),(B^c,\widetilde{T}))$ the number of copies of
$(B,T)\in \mathcal{B}$ that needs to be matched with the same
amount of copies of $(B^c,\widetilde{T})\in \mathcal{B}^c$. Then
we can only have
\begin{eqnarray*}
\ell((\{a,b\},X),(\{c,d\},\{c,d\})) & = &
\ell((\{a,b\},\{a,b,c\}),(\{c,d\},\{c,d\}))\\
 & = & \ell((\{a,b\},\{a,b,d\}),(\{c,d\},\{c,d\}))\\
 & = & \ell((\{a,b\},\{a,b\}),(\{c,d\},\{c,d\}))=2
\end{eqnarray*}
and all the others equal to $1$. So our polynomial will be
$$
\overline{M}_X(x,y)=4+6y+2x+2x^2+3x^2y+3xy+2y^2+x^2y^2+xy^2=M_X(x,y)
$$
as it should be.
\end{Ex}

Let us see why such a bijection $\psi=\psi_B$ should always
exists.

Notice that in this case we have only one basis
$B:=\{f_1,f_2,\dots,f_r\}$. Moreover, by axiom (3) for $m$, we
have that $m(B)\cdot m(B^c)=m(\emptyset)\cdot m(X)$, and hence
$m(B)$ divides $m(\emptyset)\cdot m(X)$.

Also $m(X)$ divides $m(S)$ for every $S\supseteq B$, which are
exactly the maximal rank sublists of $X$. In particular we have
$m(B)=m(X)\cdot c(B)$ for some $c(B)$.

In fact $m(X)$ divides each $\mu(S)$: recursively it divides each
$\mu(T)$ for every $T\supsetneqq S$, and it divides $m(S)$, hence
it divides $\mu(S)=m(S)-\sum_{T\supsetneqq S}\mu(T)$.

Say that $\mu(T)=m(X)\cdot a(T)$, where $a(T)\in
\mathbb{N}\cup\{0\}$. Moreover
$$
m(B)=\sum_{T\supseteq B}\mu(T)=m(X)\cdot \left(\sum_{T\supseteq
B}a(T)\right)=m(X)\cdot c(B),
$$
so that the $a(T)$'s give a partition of $c(B)$. But since
$m(B)=m(X)\cdot c(B)$ divides $m(X)\cdot m(\emptyset)$, we have
that $c(B)$ divides $m(\emptyset)$.

Dually, we have that $B^c$ is the only basis, so analogously
$m(\emptyset)=m^*(X)$ divides every $m^*(\widetilde{T})$ and hence
every
$$\mu^*(\widetilde{T}):=m^*(\widetilde{T})-\sum_{\widetilde{S}\supsetneqq
\widetilde{T}}(-1)^{|\widetilde{S}|-|\widetilde{T}|}m^*(\widetilde{S})$$
for every $\widetilde{T}\supseteq B^c$, which are the maximal rank
sublists in the dual. In particular
$m^*(B^c)=m(B)=m(\emptyset)\cdot c^*(B)=m^*(X)\cdot c^*(B)$ for
some $c^*(B)$. Say $\mu^*(\widetilde{T})=m^*(X)\cdot
a^*(\widetilde{T})=m(\emptyset)\cdot a^*(\widetilde{T})$.

Then again
$$
m(B)=m^*(B^c)=\sum_{\widetilde{T}\supseteq
B^c}\mu^*(\widetilde{T})=m^*(X)\cdot
\left(\sum_{\widetilde{T}\supseteq
B^c}a^*(\widetilde{T})\right)=m^*(X)\cdot c^*(B),
$$
which divides $m^*(X)\cdot m^*(\emptyset)=m(\emptyset)\cdot m(X)$
so that the $a^*(\widetilde{T})$'s give a partition of $c^*(B)$
which divides $m(X)=m^*(\emptyset)$.

We are now in a position to define an equidistributed matching of
the pairs in $\mathcal{B}$ with the pairs in $\mathcal{B}^*$ in
this case.

Consider a pair $(B,T)\in \mathcal{B}$, which appears with
multiplicity $\mu(T)=m(X)\cdot a(T)\neq 0$, so $a(T)\geq 1$. Now
each pair $(B^c,\widetilde{T})$ in $\mathcal{B}^*$ appears with
multiplicity $\mu^*(\widetilde{T})=m(\emptyset)\cdot
a^*(\widetilde{T})\neq 0$, so that $a^*(\widetilde{T})\geq 1$, and
the sum of these multiplicity is $m(\emptyset)\cdot c^*(B)$. Since
$c^*(B)$ divides $m(X)$, we can match our $m(X)\cdot a(T)$ many
pairs $(B,T)$ evenly among the copies of the pairs
$(B^c,\widetilde{T})$. More precisely, if $m(X)=h\cdot c^*(B)$,
then we match $h\cdot a(T)\cdot a^*(\widetilde{T})$ many copies of
$(B,T)$ with $(B^c,\widetilde{T})$, and this for all
$(B^c,\widetilde{T})\in \mathcal{B}^*$. We set
$$
\ell((B,T),(B^c,\widetilde{T})):=h\cdot a(T)\cdot
a^*(\widetilde{T}).
$$
Observe now that for distinct $T_1,T_2\in L_X$ and for any
$\widetilde{T}\in L_{X}^*$ we have
$$
\frac{\mu(T_1)}{\mu(T_2)}=\frac{m(X)\cdot a(T_1)}{m(X)\cdot
a(T_2)}=\frac{h\cdot a^*(\widetilde{T})\cdot a(T_1)}{h\cdot
a^*(\widetilde{T})\cdot a(T_2)}
=\frac{\ell((B,T_1),(B^c,\widetilde{T}))}{\ell((B,T_2),(B^c,\widetilde{T}))};
$$
also, for distinct $\widetilde{T}_1,\widetilde{T}_2\in L_{X}^*$
and for any $T\in L_X$ we have
$$
\frac{\mu^*(\widetilde{T}_1)}{\mu^*(\widetilde{T}_2)}=\frac{m(\emptyset)\cdot
a^*(\widetilde{T}_1)}{m(\emptyset)\cdot
a^*(\widetilde{T}_2)}=\frac{h\cdot a(T)\cdot
a^*(\widetilde{T}_1)}{h\cdot a(T)\cdot a^*(\widetilde{T}_2)}
=\frac{\ell((B,T),(B^c,\widetilde{T}_1))}{\ell((B,T),(B^c,\widetilde{T}_2))}.
$$
This shows that the matching $\psi=\psi_B$ that we just defined is
equidistributed.

Now that we have a good definition of our polynomial
$\overline{M}_X(x,y)$, we can turn our attention to the proof of
Theorem \ref{thm:specialcase}.

In order to prove the theorem, we will resort to a
\textit{deletion-contraction} recursion for the polynomial
$M_X(x,y)$.
\begin{Lem}
If $(\mathfrak{M}_X,m)$ is an arithmetic matroid with no proper
vectors, and $v\in X$ is a free vector, then
$$
M_X(x,y)=(x-1)M_{X_1}(x,y)+M_{X_2}(x,y),
$$
where $M_{X_1}(x,y)$ and $M_{X_2}(x,y)$ denote the arithmetic
Tutte polynomial associated to the deletion and the contraction
arithmetic matroid with respect to $v$, respectively.
\end{Lem}
Before proving the lemma, let us see an example.
\begin{Ex}
If we compute the arithmetic Tutte polynomials of the deletion and
the contraction of the vector $b$ in the Example
\ref{ex:basetorsion} we get
$$
M_{X_1}(x,y)=2x+3xy+xy^2=(y+2)(y+1)x
$$
and
$$
M_{X_2}(x,y)=4+6y+4x+6xy+2y^2+2xy^2=2(y+2)(y+1)(x+1).
$$
So
$$
(x-1)M_{X_1}(x,y)+M_{X_2}(x,y)=(y+2)(y+1)(x^2+x+1)=M_X(x,y)
$$
as it should be.
\end{Ex}
\begin{proof}[proof of the Lemma]
For $v\in A\subseteq X$ we have $rk_1(A\setminus
\{v\})=rk_2(A\setminus \{v\})=rk(A)-1$, where $rk_1$ and $rk_2$
are the rank functions of the deletion and the contraction by $v$
respectively, since $v$ is free. Hence
\begin{eqnarray*}
M_X(x,y) & = & \sum_{A\subseteq
X}m(A)(x-1)^{rk(X)-rk(A)}(y-1)^{|A|-rk(A)}\\
 & = & \sum_{v\notin A\subseteq
X}m(A)(x-1)^{rk(X)-rk(A)}(y-1)^{|A|-rk(A)}\\
 & + & \sum_{v\in A\subseteq
X}m(A)(x-1)^{rk(X)-rk(A)}(y-1)^{|A|-rk(A)}\\
 & = & \sum_{v\notin A\subseteq
X}m_1(A)(x-1)^{rk_1(X\setminus \{v\})+1-rk_1(A)}(y-1)^{|A|-rk_1(A)}\\
 & + & \sum_{v\in
A\subseteq X}m_2(A\setminus \{v\})(x-1)^{rk_2(X\setminus
\{v\})-rk_2(A\setminus
\{v\})}(y-1)^{|A\setminus \{v\}|-rk_2(A\setminus \{v\})}\\
 & = & (x-1)M_{X_1}(x,y)+M_{X_2}(x,y).
\end{eqnarray*}
\end{proof}
Notice that dually we have the following immediate corollary.
\begin{Lem}
If $(\mathfrak{M}_X,m)$ is an arithmetic matroid with no proper
vectors, and $v\in X$ is a torsion vector, then
$$
M_X(x,y)=M_{X_1}(x,y)+(y-1)M_{X_2}(x,y),
$$
where $M_{X_1}(x,y)$ and $M_{X_2}(x,y)$ denote the arithmetic
Tutte polynomial associated to the deletion and the contraction
arithmetic matroid with respect to $v$, respectively.
\end{Lem}
We will now prove that the polynomial $\overline{M}_X(x,y)$
satisfies the same recurrences.
\begin{Lem}
If $(\mathfrak{M}_X,m)$ is an arithmetic matroid with no proper
vectors, and $v\in X$ is a free vector, then
$$
\overline{M}_X(x,y)=(x-1)\overline{M}_{X_1}(x,y)+\overline{M}_{X_2}(x,y),
$$
where $\overline{M}_{X_1}(x,y)$ and $\overline{M}_{X_2}(x,y)$
denote the arithmetic Tutte polynomial associated to the deletion
and the contraction arithmetic matroid with respect to $v$,
respectively.
\end{Lem}
\begin{proof}
We have
\begin{eqnarray*}
\overline{M}_X(x,y) & = & \sum_{(B^c,\widetilde{T})\in
\mathcal{B}^*}x^{e^*(B^c,\widetilde{T})}y^{e(\psi^{-1}(B^c,\widetilde{T}))}\\
 & = & \mathop{\sum_{(B^c,\widetilde{T})\in
\mathcal{B}^*}}_{v\in
\widetilde{T}}x^{e^*(B^c,\widetilde{T})}y^{e(\psi^{-1}(B^c,\widetilde{T}))}
+\mathop{\sum_{(B^c,\widetilde{T})\in \mathcal{B}^*}}_{v\notin
\widetilde{T}}x^{e^*(B^c,\widetilde{T})}y^{e(\psi^{-1}(B^c,\widetilde{T}))}.
\end{eqnarray*}
Since $v$ is a free vector, it is contained in $B$. So it never
acts on $B$. Dually, $v$ is torsion in the dual and $v\notin B^c$,
hence it is always externally active on $B^c$.

Let us consider the elements of $\mathcal{B}^*$ involved in the
left summand. We clearly have
\begin{eqnarray*}
\mathop{\sum_{(B^c,\widetilde{T})\in \mathcal{B}^*}}_{v\in
\widetilde{T}}x^{e^*(B^c,\widetilde{T})}y^{e(\psi^{-1}(B^c,\widetilde{T}))}
& = & x\cdot \mathop{\sum_{(B^c,\widetilde{T})\in
\mathcal{B}^*}}_{v\in
\widetilde{T}}x^{e^*(B^c,\widetilde{T}\setminus
\{v\})}y^{e(\psi^{-1}(B^c,\widetilde{T}))}.
\end{eqnarray*}
We want to show that
\begin{equation} \label{eq:lemmamol}
\mathop{\sum_{(B^c,\widetilde{T})\in \mathcal{B}^*}}_{v\in
\widetilde{T}}x^{e^*(B^c,\widetilde{T}\setminus
\{v\})}y^{e(\psi^{-1}(B^c,\widetilde{T}))} =
\sum_{(B^c,\widetilde{S})\in
\mathcal{B}_{1}^*}x^{e^*(B^c,\widetilde{S})}y^{e(\psi_{1}^{-1}(B^c,\widetilde{S}))},
\end{equation}
where $\mathcal{B}_{1}^*$ denotes the list of pairs corresponding
to (the dual of) the deletion of $v$, and
$\psi_{1}^{-1}=\psi_{B\setminus \{v\}}^{-1}$ is the bijection
between $\mathcal{B}_{1}^*$ and $\mathcal{B}_{1}$ in the deletion
of $v$. If we can show this, then we have
\begin{eqnarray*}
x\cdot \mathop{\sum_{(B^c,\widetilde{T})\in \mathcal{B}^*}}_{v\in
\widetilde{T}}x^{e^*(B^c,\widetilde{T}\setminus
\{v\})}y^{e(\psi^{-1}(B^c,\widetilde{T}))} & = & x\cdot
\sum_{(B^c,\widetilde{S})\in
\mathcal{B}_{1}^*}x^{e^*(B^c,\widetilde{S})}y^{e(\psi_{1}^{-1}(B^c,\widetilde{S}))}\\
& = & x\overline{M}_{X_1}(x,y).
\end{eqnarray*}

We make two remarks. First of all, for $(B^c,\widetilde{T})\in
\mathcal{B}^*$ with $v\in \widetilde{T}$ we have
\begin{eqnarray*}
\mu_{1}^*(\widetilde{T}\setminus \{v\}) & = & \mu_{X\setminus
\{v\}}^*(\widetilde{T}\setminus
\{v\})=\sum_{\widetilde{T}\setminus \{v\}\subseteq A\subseteq
X\setminus \{v\}}(-1)^{|A|-|\widetilde{T}\setminus \{v\}|}m_{1}^*(A)\\
& = & \sum_{\widetilde{T}\setminus \{v\}\subseteq A\subseteq
X\setminus \{v\}}(-1)^{|A|-|\widetilde{T}\setminus
\{v\}|}m((X\setminus \{v\})\setminus A)\\
 & = & \sum_{\widetilde{T}\setminus \{v\}\subseteq A\subseteq
X\setminus \{v\}}(-1)^{|A|-|\widetilde{T}\setminus \{v\}|}m(X\setminus (A\cup \{v\}))\\
 & = &   \sum_{\widetilde{T}\subseteq A'\subseteq X}(-1)^{|A'|-|\widetilde{T}|}m(X\setminus A')\\
 & = & \sum_{\widetilde{T}\subseteq A'\subseteq X}(-1)^{|A'|-|\widetilde{T}|}m^*(A') = \mu_{X}^*(\widetilde{T})=
\mu^*(\widetilde{T}).
\end{eqnarray*}
Moreover, there exists a positive integer $h$ such that for every
$(B,T)\in \mathcal{B}$ we have
\begin{equation} \label{eq:h}
h\cdot \mu_{1}(T\setminus \{v\})= h\cdot \mu_{X\setminus
\{v\}}(T\setminus  \{v\})= \mu_{X}(T)= \mu(T).
\end{equation}
In fact, by axiom (2) for $m$, $m(X)=h\cdot m(X\setminus \{v\})$
for some positive integer $h$. Then, by axiom (3) for $m$, for all
$T\subseteq A\subseteq X$,
$$
m(X)\cdot m(A\setminus \{v\})=m(X\setminus \{v\})\cdot m(A),
$$
which implies that $m(A)=h\cdot m(A\setminus \{v\})$. This
immediately gives \eqref{eq:h}.

The first remark guarantees that both sides of \eqref{eq:lemmamol}
have the same number of summands. The second one guarantees that
the restriction of $\psi^{-1}$ to the elements in the left hand
side of \eqref{eq:lemmamol} is still equidistributed, and hence
guarantees that it is in fact $\psi_{1}^{-1}$. Then we already
observed that the local activities correspond in the right way.

For the other summand we have
\begin{eqnarray*}
\mathop{\sum_{(B^c,\widetilde{T})\in \mathcal{B}^*}}_{v\notin
\widetilde{T}}x^{e^*(B^c,\widetilde{T})}y^{e(\psi^{-1}(B^c,\widetilde{T}))}
 & = & \sum_{(B^c,\widetilde{T})\in \mathcal{B}^*}x^{e^*(B^c,\widetilde{T}\setminus
\{v\})}
 y^{e(\psi^{-1}(B^c,\widetilde{T}))}\\
  & - &
\mathop{\sum_{(B^c,\widetilde{T})\in \mathcal{B}^*}}_{v\in
\widetilde{T}}x^{e^*(B^c,\widetilde{T}\setminus
\{v\})}y^{e(\psi^{-1}(B^c,\widetilde{T}))}.
\end{eqnarray*}
From what we have just seen, the second summand is clearly
$-\overline{M}_{X_1}(x,y)$. For the first summand, notice that in
the contraction of $v$ we have $m_2(A)=m(A\cup \{v\})$ for every
$A\subseteq X_2=X\setminus \{v\}$. Since every maximal rank
sublist $S$ of $X$ contains $v$, its multiplicity is the same as
the multiplicity of $S\setminus\{v\}$ in the contraction of $v$,
and these are exactly the maximal rank sublists of
$X_2=X\setminus\{v\}$. For $(B^c,\widetilde{S})\in
\mathcal{B}_{2}^*$ (remember $\widetilde{S}\subseteq
X_2:=X\setminus \{v\}$) we have
\begin{eqnarray*}
\mu^*(\widetilde{S})+\mu^*(\widetilde{S}\cup\{v\}) & = &
\sum_{\widetilde{S}\subseteq A\subseteq
X}(-1)^{|A|-|\widetilde{S}|}m^*(A)+ \sum_{\widetilde{S}\cup\{v\}
\subseteq A\subseteq X}(-1)^{|A|-|\widetilde{S}\cup\{v\}|}m^*(A)\\
 & = & \sum_{\widetilde{S}\subseteq A\subseteq
X}(-1)^{|A|-|\widetilde{S}|}m^*(A)-
\mathop{\sum_{\widetilde{S}\subseteq A\subseteq X}}_{v\in
A}(-1)^{|A|-|\widetilde{S}|}m^*(A)\\
 & = & \mathop{\sum_{\widetilde{S}\subseteq A\subseteq X}}_{v\notin
A}(-1)^{|A|-|\widetilde{S}|}m^*(A) = \sum_{\widetilde{S}\subseteq A\subseteq X_2}(-1)^{|A|-|\widetilde{S}|}m(X\setminus A)\\
 & = &\sum_{\widetilde{S}\subseteq A\subseteq X_2}(-1)^{|A|-|\widetilde{S}|}m((X_2\setminus A)\cup
 \{v\})\\
 & = & \sum_{\widetilde{S}\subseteq A\subseteq X_2}(-1)^{|A|-|\widetilde{S}|}m_2(X_2\setminus A)\\
 & = &\sum_{\widetilde{S}\subseteq A\subseteq X_2}(-1)^{|A|-|\widetilde{S}|}m_{2}^*(A)=\mu_{2}^*(\widetilde{S}).
\end{eqnarray*}

Notice also that we are taking the sum over all the pairs, but we
are computing the external activity $e^*$ by removing $v$ from the
elements of $L_{X}^*$. So we get the same result as if we did it
in the contraction of $v$. Finally, the bijection $\psi$ is
clearly equidistributed, hence it corresponds to the bijection
$\psi_2$ of the contraction. Therefore
$$
\sum_{(B^c,\widetilde{T})\in
\mathcal{B}^*}x^{e^*(B^c,\widetilde{T}\setminus \{v\})}
 y^{e(\psi^{-1}(B^c,\widetilde{T}))}=\sum_{(B^c,\widetilde{S})\in
\mathcal{B}_{2}^*}x^{e^*(B^c,\widetilde{S})}
 y^{e(\psi_{2}^{-1}(B^c,\widetilde{S}))}=\overline{M}_{X_2}(x,y).
$$
This concludes the proof of the lemma.
\end{proof}
Again, dually we have the following immediate corollary.
\begin{Lem}
If $(\mathfrak{M}_X,m)$ is an arithmetic matroid with no proper
vectors, and $v\in X$ is a torsion vector, then
$$
\overline{M}_X(x,y)=\overline{M}_{X_1}(x,y)+(y-1)\overline{M}_{X_2}(x,y),
$$
where $\overline{M}_{X_1}(x,y)$ and $\overline{M}_{X_2}(x,y)$
denote the arithmetic Tutte polynomial associated to the deletion
and the contraction arithmetic matroid with respect to $v$,
respectively.
\end{Lem}
\begin{Rem}
Notice that in order to prove these lemmas we didn't use any
specific order on the list of vectors $X$: this is because the
only elements that could be acting on the basis were torsion, and
hence they were acting anyway, no matter the order on $X$. In the
general case, to deal with the proper vectors an order will be
useful.
\end{Rem}
We can now prove Theorem \ref{thm:specialcase}.
\begin{proof}[proof of Theorem \ref{thm:specialcase}]
Applying the previous Lemmas and the observation that free and
torsion elements remains free and torsion respectively under
deletion and contraction, we reduce the problem to check the
equality $M_X(x,y)=\overline{M}_X(x,y)$ in the case of an empty
list. But in this case both polynomials are obviously equal to
$m(\emptyset)$.
\end{proof}

\section{The general case}

For the general case, we want to define for each basis $B$ in $X$
a matching $\psi_B$ of the pairs in $\mathcal{B}$ of the form
$(B,T)$ with the pairs in $\mathcal{B}^*$ of the form
$(B^c,\widetilde{T})$, and then ``join'' them together in a
matching $\psi$ from $\mathcal{B}$ to $\mathcal{B}^*$.

First of all we fix a total order on the elements of $X$.

For each basis $B$ in $X$, denote by $\mathcal{B}_{B}$ the sublist
of pairs of $\mathcal{B}$ whose first coordinate is $B$. For each
pair $(B,T)$ in this list, we ignore the elements of $T$ that are
not externally active on $B$, identifying the pairs that differ
only for such nonactive elements. We do the same for $B^c$ in
$\mathcal{B}_{B^c}^*$. We claim that we can match evenly these
pairs and we call $\psi_B$ such an equidistributed matching. Then
$\psi$ will be just the ``join'' of these matchings.
\begin{Rem}
Notice that in general the matching $\psi$ depends on the order
that we choose.
\end{Rem}

Before showing that such a matching $\psi$ exists, we show
how we will use it.

We define
$$
\overline{M}_X(x,y):=\sum_{(B,T)\in
\mathcal{B}}x^{e^*(\psi(B,T))}y^{e(B,T)},
$$
where $\psi$ is the matching that we just described.

The following theorem is the main result of this paper.

\begin{Thm} \label{thm:generalcase}
If $(\mathfrak{M}_X,m)$ is an arithmetic matroid, then
$$
M_X(x,y)=\overline{M}_X(x,y)=\sum_{(B,T)\in
\mathcal{B}}x^{e^*(\psi(B,T))}y^{e(B,T)},
$$
where $\psi$ is the bijection between $\mathcal{B}$ and
$\mathcal{B}^*$ described above.
\end{Thm}
\begin{Rem}
Notice that, even if in general the matching $\psi$ depends on the
order that we choose, the polynomial
$\overline{M}_X(x,y)=M_X(x,y)$ will not depend on it.

Also, observe that in the special case when all the multiplicities
are equal to $1$ the arithmetic Tutte polynomial reduces just to
the classical Tutte polynomial of the matroid. Moreover in this
case $L_X=(X)=L_{X^*}$, and our combinatorial description
corresponds exactly to the one given by Crapo in Theorem
\ref{Cra}. In this sense our result can be seen as a
generalization of Crapo's Theorem.
\end{Rem}
Let us first look at an example.
\begin{Ex}
Let
$X=\{a:=(3,3,0),b:=(-6,-6,-6),c:=(0,0,3),d:=(0,0,12)\}\subseteq
\mathbb{Z}^3$. Here we take as the group $G$ the maximal subgroup
of $\mathbb{Z}^3$ in which $\langle X\rangle$ has finite index.
The rank of the arithmetic matroid $\mathfrak{M}_X$ associated to
the pair $(G,X)$ (i.e. of $\langle X\rangle$) is $2$, and the
bases are $\{a,b\}$, $\{a,c\}$, $\{a,d\}$, $\{b,c\}$ and
$\{b,d\}$.

The multiplicities are $m(\emptyset)=1$,
$m(\{a\})=m(\{c\})=m(\{c,d\})=3$, $m(\{b\})=6$,
$m(\{a,c\})=m(\{a,b,c\})=m(\{a,c,d\})=m(X)=9$, $m(\{d\})=12$,
$m(\{a,b\})=m(\{b,c\})=m(\{a,b,d\})=m(\{b,c,d\})=18$,
$m(\{a,d\})=36$, $m(\{b,d\})=72$.

We have
$$
L_{X}=(X^9,\{b,c,d\}^9,\{a,b,d\}^9,\{a,d\}^{18},\{b,d\}^{45})
$$
and
\begin{eqnarray*}
L_{X}^* & = &
(X,\{b,c,d\}^{2},\{a,c,d\}^5,\{a,b,d\}^2,\{a,b,c\}^{11},
\{a,c\}^{55},\{a,d\}^{10}, \{b,c\}^{22},\\
 & & \{b,d\}^{4},\{c,d\}^{10}).
\end{eqnarray*}
We set the order $a<b<c<d$, so in $\mathcal{B}$ we have
$e(\{a,d\},X)=e(\{b,c\},X)=e(\{b,d\},\{b,c,d\})=e(\{b,d\},\{a,b,d\})=1$,
$e(\{b,d\},X)=2$ and all the others are $0$, while in
$\mathcal{B}^*$ we have
$e^*(\{b,c\},X)=e^*(\{b,c\},\{a,b,c\})=e^*(\{b,d\},\{a,b,d\})=
e^*(\{b,d\},X)=e^*(\{c,d\},\{a,c,d\})=e^*(\{c,d\},\{b,c,d\})=1$,
$e^*(\{c,d\},X)=2$ and all the others are $0$.

For example, for the basis $\{a,d\}$ we distinguish $X^9$, which
contain the active vector $c$, from $\{a,b,d\}^9$ and
$\{a,d\}^{18}$, which don't contain any active vector. Hence we
have $9$ pairs which give a $y$ and $27$ pairs which give a $1$ to
match. In the dual, for the basis $\{b,c\}$, we distinguish $X$
and $\{a,b,c\}^{11}$, which contain the active vector $a$, from
$\{b,c,d\}^2$ and $\{b,c,d\}^{22}$. Hence we have $12$ pairs which
give a $x$ and $24$ pairs which give a $1$ to match. Therefore we
have a summand $3xy+6y+9x+18=3(x+2)(y+3)$.

Doing the same with all the other bases we get
$$
\overline{M}_X(x,y)=x^2+19x+88+3xy+33y+9y^2=M_X(x,y),
$$
as it should be.
\end{Ex}
Let us now prove the existence of the matching.
\begin{Lem} \label{lem:equi}
The equidistributed matchings $\psi_B$ defined before exist.
\end{Lem}
\begin{proof}
To show that such matchings $\psi_B$ exist we apply the following
algorithm. Recall that we fixed an order on $X$.

If $X$ has no proper vectors, we use the matching $\psi_B$ that we
constructed in the molecular case.

If not, we look at the greatest proper vector $v$ in $X$. There
are two cases: $v\in B$ or $v\notin B$.

In the first case we do a contraction of $v$, carrying the order
on $X$ over on $X_{2}=X\setminus \{v\}$. Notice that in this case
$B\setminus \{v\}$ is still a basis, and what was active on $B$ is
now active on $B\setminus \{v\}$, while what was not active on $B$
is not active on $B\setminus \{v\}$. Moreover, since
$m_2(A)=m(A\cup\{v\})$ for $A\subseteq X_2$, for $S\subseteq X_2$
such that $S\supseteq B\setminus \{v\}$ we have
$\mu_2(S)=\mu(S\cup\{v\})$. So the list
$(\mathcal{B}_2)_{B\setminus \{v\}}$, where $\mathcal{B}_2$ is as
usual the list of the contraction, may be derived from
$\mathcal{B}_B$ by removing $v$ from both elements of each pair.

In the second case, notice that $v$ is not active on $B$. This
time we make a deletion of $v$. Notice that $B$ is still a basis
in the deletion matroid. Moreover what was active on $B$ it
remains active, and what was not active on $B$ it remains
nonactive. Also, if $S\subseteq X_1:=X\setminus\{v\}$ is such that
$S\supseteq B$, then
\begin{eqnarray*}
\mu(S)+\mu(S\cup\{v\}) & = & \sum_{T\supseteq
S}(-1)^{|T|-|S|}m(T)+ \sum_{T\supseteq
S\cup\{v\}}(-1)^{|T|-|S\cup\{v\}|}m(T)\\
 & = & \sum_{T\supseteq
S}(-1)^{|T|-|S|}m(T)- \mathop{\sum_{T\supseteq S}}_{v\in
T}(-1)^{|T|-|S|}m(T)\\
 & = & \mathop{\sum_{T\supseteq S}}_{v\notin
T}(-1)^{|T|-|S|}m(T)=\mu_1(S).
\end{eqnarray*}
So $(\mathcal{B}_1)_{B\setminus \{v\}}$, where $\mathcal{B}_1$ is
the list corresponding to the deletion, can be obtained again from
$\mathcal{B}_B$ by removing $v$ from both elements (in fact in the
first does not appear) of its pairs.

Notice that these cases are dual of each other, meaning that a
proper vector in $B$ corresponds to a proper vector not in $B^c$
in the dual, and viceversa.

So in both cases what we are really doing is to ignore the element
$v$, either if it is in the basis (and hence dually nonactive) or
it is nonactive (and hence dually in the basis).

We iterate this procedure until we get a molecule. But in this
case we can implement the equidistributed bijection that we
constructed in the previous section: this is going to be our
$\psi_B$.
\end{proof}

We now want to prove Theorem \ref{thm:generalcase}. We will use
the appropriate deletion-contraction recursion for the proper
vectors.

\begin{Lem}
If $(\mathfrak{M}_X,m)$ is an arithmetic matroid with, and $v\in
X$ is a proper vector, then
$$
M_X(x,y)=M_{X_1}(x,y)+M_{X_2}(x,y),
$$
where $M_{X_1}(x,y)$ and $M_{X_2}(x,y)$ denote the arithmetic
Tutte polynomial associated to the deletion and the contraction
arithmetic matroid with respect to $v$, respectively.
\end{Lem}
This lemma has been proved in \cite{MoT}. We repeat here the proof
for completeness.
\begin{proof}
We have
\begin{eqnarray*}
M_X(x,y) & = & \sum_{A\subseteq
X}m(A)(x-1)^{rk(X)-rk(A)}(y-1)^{|A|-rk(A)}\\
 & = & \sum_{v\notin A\subseteq
X}m(A)(x-1)^{rk(X)-rk(A)}(y-1)^{|A|-rk(A)}\\
 & + & \sum_{v\in A\subseteq
X}m(A)(x-1)^{rk(X)-rk(A)}(y-1)^{|A|-rk(A)}\\
 & = & \sum_{v\notin A\subseteq
X}m_1(A)(x-1)^{rk_1(X\setminus \{v\})-rk_1(A)}(y-1)^{|A|-rk_1(A)}\\
 & + & \sum_{v\in
A\subseteq X}m_2(A\setminus \{v\})(x-1)^{rk_2(X\setminus
\{v\})-rk_2(A\setminus
\{v\})}(y-1)^{|A\setminus \{v\}|-rk_2(A\setminus \{v\})}\\
 & = & M_{X_1}(x,y)+M_{X_2}(x,y).
\end{eqnarray*}
\end{proof}
\begin{Ex}
Let $X=\{a:=(2,-1),b:=(-1,2),c:=(1,1)\}\subseteq \mathbb{Z}^2$. If
we compute the arithmetic Tutte polynomials associated to the
deletion and contraction of the proper vector $c$ we get
$$
M_{X_1}(x,y)=x^2+2\text{ and } M_{X_2}(x,y)=x+2+3y,
$$
so that
$$
M_{X_1}(x,y)+ M_{X_2}(x,y)=x^2+x+4+3y =M_X(x,y),
$$
as it should be.
\end{Ex}
We will now prove that the polynomial $\overline{M}_X(x,y)$
satisfies the same recurrence.
\begin{Lem}
Let $(\mathfrak{M}_X,m)$ be an arithmetic matroid, and let us fix
an order on the vectors $X$. If $v\in X$ is the greatest proper
vector, then
$$
\overline{M}_X(x,y)=\overline{M}_{X_1}(x,y)+\overline{M}_{X_2}(x,y),
$$
where $\overline{M}_{X_1}(x,y)$ and $\overline{M}_{X_2}(x,y)$
denote the arithmetic Tutte polynomial associated to the deletion
and the contraction arithmetic matroid with respect to $v$,
respectively.
\end{Lem}
\begin{proof}
We have
\begin{eqnarray*}
\overline{M}_X(x,y) & = & \sum_{(B,T)\in
\mathcal{B}}x^{e^*(\psi(B,T))}y^{e(B,T)}= \sum_{T\in L_X}
\mathop{\sum_{B\subseteq T}}_{B\text{ basis}}x^{e^*(\psi(B,T))}y^{e(B,T)}\\
 & = & \sum_{T\in L_X}
\mathop{\sum_{B\subseteq T}}_{v\notin B\text{
basis}}x^{e^*(\psi(B,T))}y^{e(B,T)}+\sum_{T\in L_X}
\mathop{\sum_{B\subseteq T}}_{v\in B\text{
basis}}x^{e^*(\psi(B,T))}y^{e(B,T)}.
\end{eqnarray*}
In the first summand we take all the bases not containing $v$, and
when we compute the statistics, since $v$ is the greatest vector,
it does not act externally on these bases. Moreover, in the dual
it is contained in every basis $B^c$, so it does not act
externally on them too. Also, recall from the proof of Lemma
\ref{lem:equi} that, for the bases involved in this summand, the
bijections $\psi_B$ of which $\psi$ is made up of correspond
exactly to the bijections in the deletion of $v$. Therefore this
summand is the same as the polynomial of the deletion
$\overline{M}_{X_1}(x,y)$.

For the second summand, since $v$ is proper also in the dual, we
can make the dual reasoning, getting that this is the polynomial
$\overline{M}_{X_2}(x,y)$ of the contraction by $v$. This
completes the proof.
\end{proof}
We can now prove Theorem \ref{thm:generalcase}.
\begin{proof}[proof of Theorem \ref{thm:generalcase}]
Fix an order on the vectors $X$. Applying iteratively this last
recurrence to the greatest proper vectors, we can reduce the
problem to the molecular case.

But this case is the content of Theorem \ref{thm:specialcase}.
\end{proof}

\section{A remark on log-concavity}

We recall some well-known definitions. A sequence of positive
integers $\{a_m\}$ is
\begin{itemize}
  \item \emph{unimodal} if $a_1\leq\dots\leq a_{k-1}\leq \ a_k\geq
  a_{k+1}\geq a_m$ for some $k$;
  \item \emph{log-concave} if $a_k^2\geq a_{k-1}a_{k+1}$ for every $k$.
\end{itemize}
It is easy to see that log-concavity implies unimodality. We say
that a polynomial in one variable is log-concave (resp. unimodal)
if the sequence of (the absolute values of) its coefficients is.
Log-concavity problems are widely studied in combinatorics: the
reader can refer to the surveys \cite{StL1}, \cite{BrLC},
\cite{StL2}.

In the early '70s Rota, Heron and Welsh (\cite{Rot}, \cite{Her},
\cite{Wel}) conjectured that the characteristic polynomial of a
hyperplane arrangement is log-concave. Recently, a proof has been
proposed in \cite{Huh}.

Another famous conjecture is the following one. Let $\mathfrak{M}$
be a matroid on a list $X$, and $i_k$ be the number of its
independent sublists of rank $k$. In \cite{Ma} Mason conjectured
that the sequence $\{i_k\}$ is log-concave. Mason's conjecture has
been recently proved by Matthias Lenz in \cite{Le2}.

Notice that the above statements can be rephrased in terms of the
Tutte polynomial. Namely the Rota-Heron-Welsh conjecture claims
that for any ($0$-representable) matroid $\mathfrak{M}$,
$T(\mathfrak{M};1-q,0)$ is log-concave. On the other hand, Mason's
conjecture claims that $T(\mathfrak{M};q+1,1)$ is log-concave.

The following example, which has been suggested to us by Matthias
Lenz, shows that the corresponding evaluations of the arithmetic
Tutte polynomial are not log-concave. Notice that by Theorem
\ref{thm:Luca}, this implies that the characteristic polynomial of
a toric arrangement is not necessarily log-concave.
\begin{Ex}
Let $X:=\{(1,0,0,0),(0,1,0,0),(0,0,1,0),(1,1,1,5)\}\subseteq
\mathbb{Z}^4$. Then
$$
M_X(1-q,0)=5-4q+6q^2-4q^3+q^4
$$
and
$$
M_X(1+q,1)=5+4q+6q^2+4q^3+q^4
$$
are not unimodal.
\end{Ex}

\end{document}